\documentclass[aop]{imsart}
 \setattribute{journal}{name}{}
\RequirePackage[OT1]{fontenc}
\RequirePackage{amsthm,amsmath}
\RequirePackage[numbers]{natbib}
\RequirePackage[colorlinks,citecolor=blue,urlcolor=blue]{hyperref}
\usepackage[T1]{fontenc}
\usepackage{amssymb}
\usepackage{pdfsync}
\usepackage{dsfont}
\numberwithin{equation}{section}
\newtheorem{thm}{Theorem}[section]
\newtheorem{prop}[thm]{Proposition}
\newtheorem{lemma}[thm]{Lemma}
\newtheorem{cor}[thm]{Corollary}
\theoremstyle{definition}
\newtheorem{dfn}{Definition}[section]

\newtheorem{exm}{Example}[section]
\theoremstyle{remark}
\newtheorem{rmk}{Remark}[section]
\allowdisplaybreaks
\def\N{\mathbb{N}} 
\def\Z{\mathbb{Z}}
\def\R{\mathbb{R}} 
\def\C{\mathbb{C}} 
\def\1{~~\mbox{I\hspace{-.6em}1}} 
\def\Lip{\mbox{Lip\,}}
\arxiv{arXiv:1207.4951}

\startlocaldefs
\numberwithin{equation}{section}
\theoremstyle{plain}
\endlocaldefs

\begin{document}

\begin{frontmatter}
\title{Weak transport inequalities and applications to exponential and oracle inequalities}
\runtitle{Weak transport inequalities and applications}

\begin{aug}
\author{\fnms{Olivier} \snm{Wintenberger}\ead[label=e1]{olivier.wintenberger@upmc.fr}\ead[label=u1,url]{http://wintenberger.fr}}

\runauthor{O. Wintenberger}

\affiliation{Sorbonne Universit\'es, UPMC Univ Paris 06, EA 3124 LSTA}

\address{LSTA, University Piere et Marie Curie, Case 158\\
4 place Jussieu\\
75005 Paris\\
FRANCE
\printead{e1}\\
\printead{u1}}
\end{aug}

\begin{abstract}
We extend the dimension free Talagrand inequalities for convex distance \cite{talagrand:1995} using an extension of Marton's weak transport \cite{marton:1996a} to other metrics than the Hamming distance. We study the dual form of these weak transport inequalities for the euclidian norm and prove that it implies sub-gaussianity and convex Poincar\'e inequality \cite{bobkov:gotze:1999a}. We obtain new weak transport inequalities for non products measures  extending the results of Samson in \cite{samson:2000}. Many examples are provided to show that the euclidian norm is an appropriate metric for  classical time series. Our approach, based on trajectories coupling, is more efficient to obtain dimension free concentration than existing contractive assumptions  \cite{djellout:guillin:wu:2004,marton:2004}. Expressing the concentration properties of the ordinary least square estimator as a conditional weak transport problem, we derive   new oracle inequalities with fast rates of convergence in dependent settings.
\end{abstract}



\end{frontmatter}

\section{Introduction}\label{sec1}
In his remarkable paper \cite{talagrand:1995}, Talagrand proves that convex distances have dimension free concentration properties. 
Since the seminal work of Marton \cite{marton:1986}, transport inequalities are known to efficiently yield such dimension free concentration inequalities. Using a duality argument, Bobkov and Gotze \cite{bobkov:gotze:1999} even proved that transport inequalities are equivalent to some concentration inequalities. Our references on the subject are the monograph of Villani \cite{villani:2009}, the survey of Gozlan and Leonard \cite{gozlan:leonard:2010} and the textbook of Boucheron {\it et al.} \cite{boucheron:lugosi:massart:2013} for the statistical applications. In dependent settings, transport inequalities appear as a nice alternative  to the modified log-Sobolev approach of Massart \cite{massart:2007} for obtaining dimension free concentration inequalities useful to obtain  fast rates of convergence in mathematical statistics. This article develops new kinds of transport inequalities, new exponential inequalities and new oracle inequalities with fast rates of convergence in dependent settings.\\

In the case of product measures with common margin (iid case), the modified log-Sobolev approach developed in \cite{massart:2007} leads to optimal dimension free concentration inequalities of Bernstein's type. However, for non product measures, such inequalities do not hold in their optimal form in many situations. The reason is the following: in the bounded iid case, Bernstein's inequality yields gaussian behavior for  deviations less than a  bound depending on the essential supremum. In many bounded Markovian cases, their exists a unique regeneration scheme of iid cycles with random length. Then the variance terms in the Bernstein's type inequalities are perturbed by the concentration properties of the random length, see  \cite{bertail:clemencon:2010}. It yields an additional  term, at least logarithmic, which cannot be removed, see \cite{adamczak:2008}. It is a drawback for statistical applications for whom the variance term of the iid case is essential. To bypass this problem, many authors assume  contractive conditions on the kernel of Markov chains, see Marton \cite{marton:1996} under geometric ergodicity and Lezaud \cite{lezaud:1998} under a spectral gap condition.  For symmetric Markov process, the spectral gap condition is more general than uniform ergodicity and is also necessary for Bernstein's inequality, see \cite{guillin:leonard:wu:yao:2009}.\\

Many classical models in time series analysis do not satisfy such contractive conditions. Fortunately, the classical Bernstein's inequality also holds for non Markovian processes under $\tilde \Gamma$-weakly dependent conditions, closely related with the uniform mixing condition, see \cite{samson:2000}. This result yields fast convergence rates of order $n^{-1}$ in oracle inequalities (comparable to those in the iid case) in dependent settings, see \cite{alquier:li:wintenberger}. However, this approach relies on the maximal coupling properties of the Hamming distance and cannot be extended to other metrics, see \cite{dedecker:prieur:raynaud:2006}. For other metrics, contractive conditions are used by Marton \cite{marton:2004} and Djellout {\it et al.} \cite{djellout:guillin:wu:2004} to extend classical dimension free transport inequalities $T_{2,d}(C)$ in a dependent context for metrics $d$ different than the Hamming one. If the "constant" $C$ in the transport inequality  is sufficiently close to the variance term then Bernstein's inequality is recovered and fast convergence rates are achieved, see \cite{joulin:ollivier:2010}. Otherwise, a tradeoff must be done between   the estimates of the variance terms and the accuracy of the contractive coupling schemes, see \cite{wintenberger:2010} for details. The fast rates of convergence in oracle inequalities are not achieved because the variance terms  of the Bernstein's types inequalities do not have the same order than in the iid case. On the contrary, the Hoeffding  inequality   is easily extended to very general dependent case using the bounded difference method, see \cite{mcdiarmid:1989,rio:2000,djellout:guillin:wu:2004}. Unfortunately, the Hoeffding inequality, equivalent to the $T_1$ transport inequality, is not dimension free and yields oracle inequalities with slow rates of convergence of order  $n^{-1/2}$, see \cite{alquier:wintenberger:2012}.\\

In this paper we develop weak transport inequalities to obtain dimension free exponential inequalities and thus fast convergence  rates in oracle inequalities. Let $E$ be a Polish space and $d$ be a lower semi-continuous metric on $E$. With the notation $P[h]=\int hdP$ for any probability measure $P$ and any measurable function $h$, we say that $P$ satisfies the weak transport inequality $\tilde T_{p,d}(C)$ for any $C>0$ and $1\le p\le q$ if for any measure $Q$
$$
\sup_{\alpha}\inf_{\pi}\frac{\pi[\alpha(Y)d(X,Y)]}{(Q[\alpha(Y)^q])^{1/q}}\le \sqrt{2C\mathcal K(Q|P)}
$$
with $1/p+1/q=1$ and the convention $+\infty/+\infty=0/0=0$.
Here $\alpha$ is any non-negative measurable function, $\pi$ is any coupling scheme of $(X,Y)$ with margins $(P,Q)$ and $\mathcal K(Q|P)$ is the relative entropy $Q[\log(dP/dQ)]$ (also called the Kullback-Leibler divergence). As the roles of $P$ and $Q$ are not the same, we also introduce $\tilde T_{p,d}^{(i)}(C)$ where $P$ and $Q$ are interchanged in the left hand side term. An application of Sion's minimax theorem shows that the weak transport inequalities  are extensions of the Marton transport inequalities  \cite{marton:1996a} 
$$
\inf_{\pi}\sup_{\alpha}\frac{\pi[\alpha(Y)1_{X\neq Y}]}{(Q[\alpha(Y)^q])^{1/q}}\le \sqrt{2C\mathcal K(Q|P)}.
$$ 
These inequalities are  weakened forms of the classical $T_{p,d}(C)$ transport inequality 
$$
W_{p,d}(P,Q):=\inf_{\pi }\pi[d^p(X,Y)]^{1/p}\le \sqrt{2C\mathcal K(Q|P)}.
$$
Contrary to the classical $T_{p,d}(C)$ transport inequalities, any compactly supported measure $P$ satisfies the weak $\tilde T_{p,d}(C)$ transport inequalities for any $1\le p\le 2$. Moreover, the weak transport  inequalities extend nicely to non-products non-contractive measures $P$ on $E^n$, $n\ge 1$. Using a new Markov trajectories coupling scheme, our main result in Theorem \ref{th:wti} states that there exists $C'>0$ such that
\begin{equation}\label{eq:mr}
\sup_{\alpha}\inf_{\pi}\frac{\sum_{j=1}^n\pi[\alpha_j(Y)d(X_j,Y_j)]}{(\sum_{j=1}^nQ[\alpha_j(Y)^q])^{1/q}}\le \sqrt{n^{2/p-1}C'\mathcal K(Q|P)}.
\end{equation}
The main assumptions hold on the conditional laws $P_{|x^{(i)}}$ of the trajectory $(X_{i+1},\ldots,X_n)$ given that $(X_i,\ldots,X_0)=x^{(i)}=(x_i,\ldots,x_0)$. Fix a lower semi-continuous auxiliary metric $d'$ satisfying $d'\ge Md$. We assume the existence of a trajectory coupling scheme $\pi_{|i}$ of the $E^{n-i}$-supported measures $(P_{|x^{(i)}},P_{|y_i,x^{(i-1)}})$ such that 
$$
\pi_{|i}(d^p(X_k,Y_k))^{1/p}\le \gamma_{k,i}(p)\,d'(x_i,y_i),\qquad \forall\, i< k\le n.
$$
The existence of such coefficients $\gamma_{k,i}(p)$ for any $1\le i<k\le  n$ is called the $\Gamma(p)$-weakly dependent condition.
When $d=d'$ is the Hamming distance, the $\Gamma(2)$-weak dependance coincides with the weak dependence already studied by Samson \cite{samson:2000} and we recover his results. We keep the notation of \cite{samson:2000} and denote $\tilde\Gamma(p)$ the weak dependence  when $d$ is the Hamming distance. However, to deal with more general and classical time series, we prefer to choose $d$ as the euclidian norm, see Section \ref{sec:ex}. When $p=1$ and $\tilde T_{1,d}(C)=\tilde T_{1,d}^{(i)}(C)=T_{1,d}(C)$ by definition, the $\Gamma(1)$-weak dependence coincides with the one of \cite{rio:2000} when $d'$ is the Hamming distance and the one of \cite{djellout:guillin:wu:2004} when $d'=d$. Thus we recover the Hoeffding's inequalities of \cite{rio:2000,djellout:guillin:wu:2004}. They are not dimension free  because $n^{2/p-1}=n$ as $p=1$ and we prefer to focus our study on the case $p=2$. We then prove dimension free concentration for ARMA processes under the minimal dependence assumption that the stationary distribution exists. Our approach improves the existing ones based on contractive arguments \cite{djellout:guillin:wu:2004,marton:2004} for classical time series; for instance, considering the Markov chain $(X_t,\xi_t)$ formed by an ARMA(1,1) process $X_{t}=\phi X_{t-1}+\xi_t+\theta \xi_{t-1}$, the contraction condition is $\phi^2 +\theta^2<1$ whereas our condition is $|\phi|<1$. Thus, we answer positively to a crucial question raised in Remark 3.6 in \cite{djellout:guillin:wu:2004}.\\

Weakening transport inequalities does not deteriorate the concentrations properties used in mathematical statistics. More precisely, we show that $\tilde T_{2,d}(C)$ and $\tilde T^{i}_{2,d}(C)$ yields the convex distance dimension free estimate due to Talagrand:
$$
P\Big[\exp\Big(\frac{d_c^2(X,A)}{4C}\Big)\Big]\le \frac1{P(A)}, \mbox{ for any measurable set }A.
$$
Here $d_c(x,A)$ is the convex distance of Talagrand \cite{talagrand:1995} when $d$ is the Hamming distance and the euclidian distance to the convex hull of $A$ as in Maurey \cite{maurey:1991} when $d$ is the euclidian norm. Following Bobkov and Ledoux \cite{bobkov:ledoux:1997}, we obtain this result by analyzing the dual form of the weak transport inequality. If $P$ satisfies $\tilde T_{2,d}(C)$ on $E $ then for any function $f$   such that there exist a function $\alpha(x)$ satisfying $f(x)-f(y)\le\alpha(x)d(x_j,y_j)$ for any $x,y\in E^2$ we have
\begin{equation}\label{eq:ni}
P\Big[\exp\Big(\lambda (f-P[f])-\frac {C\lambda^2}2 \alpha^2  \Big)\Big)\Big]\le 1,\qquad \lambda>0.
\end{equation}
When $d$ is the Hamming distance,  inequality \eqref{eq:ni} yields to the Bernstein inequality, see Ledoux \cite{ledoux:1996} in the independent setting and Samson \cite{samson:2000} in the $\tilde\Gamma(2)$-weakly dependent setting. When the function $f$ is a convex function, the condition above is automatically satisfied with $L_j=\partial_j f$ (the sub-gradients) and $d$ the euclidian norm. The inequality \eqref{eq:ni} coincides with generalizations of the Tsirel'son inequality discovered in \cite{tsirelson:1985} for gaussian measures, see \cite{bobkov:gentil:ledoux:2001}. Using the dual form \eqref{eq:ni}, we also prove that $\tilde T_2$ implies  sub-gaussianity and convex Poincar\'e inequality \cite{bobkov:gotze:1999a}. Then, the weak transport approach provides dimension free concentration properties of ARMA processes under minimal assumptions and  is sufficient for statistical application.\\

As the transport inequalities yield concentration of measures via relative entropy, we couple it with the statistical  PAC-bayesian paradigm that describes the accuracy of estimators in term of relative entropy too, see \cite{mcallester:1999}. Thus, instead of using the approach based on the supremum of the empirical process \cite{massart:2007}, we introduce the  conditional weak transport inequalities that provides sharp oracle inequalities. We apply this new approach to the Ordinary Least Square (OLS) estimator $\hat\theta$ in the linear regression context (other interesting statistical issues will be investigated in the future). Denoting by $R$ the risk of prediction, an oracle inequality states with high probability that $R(\hat\theta)\le(1+\eta)R(\overline\theta) + \Delta_n\eta^{-1_{\eta\neq 0}}$ where $\eta\ge 0$, $\overline \theta$ is the oracle defined as $R(\overline\theta) \le R(\theta)$ for all $\theta$ and $\Delta_n$ is the rate of convergence.  If $\eta=0$ then the oracle inequality is said to be exact and otherwise it is non exact, see   \cite{lecue:mendelson:2012}. The dimension free concentration properties obtained from the weak transport inequalities with $p=2$ yield to fast rates of convergence $\Delta_n\propto n^{-1}$. When $d$ is the Hamming distance, we recover in the conditional weak transport inequalities the variance terms of the iid case. These variance terms play  a crucial role through Bernstein's condition \cite{bartlett:mendelson:2006} to obtain exact oracle inequalities with fast rates of convergence. Thus, when $d$ is the Hamming distance, we obtain new exact oracle inequalities with fast convergence rates for the OLS $\hat \theta$ in the $\tilde\Gamma(2)$-weakly dependent case. However, in more general cases, Bernstein's condition cannot hold and the variance terms cannot be nicely estimated. There, we emphase the fact that the Tsirelson inequality should be preferred to the Bernstein one. Hence, using the euclidian metric, we obtain new nonexact oracle inequalities for the OLS $\hat \theta$ for $\Gamma(2)$-weakly dependent time series. The efficiency of the OLS is proved for many models such as classical ARMA models. \\

The paper is organized as follows: in Section \ref{sec:wtc} are developed the preliminaries that are used in the proof of our main result, a weak transport inequalities for non product measures stated in Section \ref{sec:wti}. We also study the dual form of the weak transport inequalities, the Tsirel'son inequality that are satisfied and the connection with Talagrand's inequalities in this Section \ref{sec:wti}. Section \ref{sec:ex} is devoted to some examples of $\tilde \Gamma$ and $\Gamma$-weakly dependent processes. Finally, new oracle inequalities with fast rates of convergence are given in Section \ref{sec:oi}.

\section{Weak transport costs, gluing lemma and Markov couplings}\label{sec:wtc}
\subsection{Weak transport costs on $E$}
Let $M(F)$ denotes the set of probability measures on some space $F$,  $M^+(F)$ the set of lower semi-continuous non negative measurable functions and $\tilde M(P,Q)$ the set of coupling measures $\pi_{x,y}$, i.e. $\pi_{x,y}\in M(E^2)$ with margins $\pi_x=P$ and $\pi_y=Q$. Let $(p,q)$ be real numbers satisfying $1\le p\le 2$ and $1/p+1/q=1$. 
Let us define the weak transport cost as
\begin{equation}
\tilde W_{p,d}(P,Q)=\sup_{\alpha \in M^+(E)}\inf_{\pi\in \tilde M(P,Q)}\frac{\pi[\alpha(Y)d(X,Y)]}{Q[\alpha^q]^{1/q}}
\end{equation}
with the classical conventions $Q[\alpha^q]^{1/q}=\mbox{ess}\sup{\alpha(Y)}$ when $q=\infty$ and  $+\infty/+\infty=0/0=0$. For fixed $\alpha\in M^+(F)$, let us denote
\begin{equation}
\tilde W_{\alpha,d} (P,Q)=\inf_{\pi \in \tilde M(P,Q)} \pi[\alpha(Y)d(X,Y)].
\end{equation}
Note that $\tilde W$ is not symmetric and that $\tilde W_{p,d}(P,Q)=\tilde W_{p,d}(Q,P)= \tilde W_{\alpha,d} (P,Q)= \tilde W_{\alpha,d} (Q,P)=0$ if $P=Q$. Note that $\alpha\in M^+$ and $d$ are assumed to be lower semi-continuous such as the optimal transport in the weak transport cost definition exists, see for example \cite{gozlan:leonard:2010}. Now let us show that the weak transport cost satisfies the triangular inequality.  It is a simple consequence of the second assertion of the following version of the gluing Lemma:
\begin{lemma}\label{lem:glu}
For any coupling $\pi_{x,y}\in\tilde M(P,Q)$ and $\pi_{y,z}\in\tilde M(Q,R)$ respectively there exists a distribution $\pi_{x,y,z}$ with  corresponding margins   and such that $X$ and $Z$ are independent conditional on $Y$, i.e. $\pi_{x,z|y}=\pi_{x|y}\pi_{z|y}$.
\end{lemma}
\begin{proof}
From the classical gluing Lemma, se for example the Villani's textbook \cite{villani:2009}, we can choose $\pi_{x,y,z}$ such that
$
\pi_{x,y,z} =\pi_{x|y} \pi_{z|y} \pi_y
$ as the margins corresponds: $\pi_{x|y}   \pi_y=\pi_{x,y}$ and $ \pi_{z|y} \pi_y=\pi_{y,z}$
The conditional independence follows from the specific form of $\pi_{x,y,z}$ as $\pi_{x,z|y}=\pi_{x,y,z}/\pi_y$ by definition.
\end{proof}
The conditional independence in the gluing Lemma \ref{lem:glu} is the main ingredient to prove the triangular inequality on $\tilde W_{p,d}$:
\begin{lemma}\label{lem:ti1}
For any $P,Q,R$ we have
\begin{equation}\label{eq:ti1}
\tilde W_{p,d}(P,R)\le \tilde W_{p,d}(P,Q)+\tilde W_{p,d}(Q,R)
\end{equation}
\end{lemma}
\begin{proof}
Let us fix $\alpha\in M^+(E)$ such that $R[\alpha^q]<\infty$.  We have
$$
\pi_{x,z}[\alpha(Z)d(X,Z)]\le \pi [\alpha(Z)d(X,Y)]+\pi_{y,z}[\alpha(Z)d(Y,Z)].
$$
Let us choose $\pi_{y,z}^\ast$ satisfying 
$$
\pi_{y,z}^\ast[\alpha(Z)d(Y,Z)]=\inf_{\pi\in \tilde M(Q,R)}  \pi[\alpha(Z)d(Y,Z)]\le  R[\alpha^q]^{1/q}\tilde W_{p,d}(Q,R).
$$
By conditional independence in Lemma \ref{lem:glu}, we also have
$$
\pi [\alpha(Z)d(X,Y)]=\pi_{x,y}[\pi_{z|y}^\ast[\alpha(Z)|Y]d(X,Y)]=:\pi_{x,y}[\tilde \alpha(Y)d(X,Y)].
$$
Let us choose $\pi_{x,y}^\ast$ satisfying
$$
\pi_{x,y}^\ast[\tilde\alpha(Y)d(X,Y)]=  \inf_{\pi\in \tilde M(P,Q)}   \pi[\tilde\alpha(Y)d(X,Y)]\le  Q[\tilde \alpha^q]^{1/q}\tilde W_{p,d}(P,Q).$$
Note that $Q[\tilde \alpha^q]=Q[\pi_{z|y}^\ast[\alpha(Z)|Y]^q]\le R[\alpha^q]$ using Jensen's inequality. Let us denote $\pi^\ast=\pi^\ast_{x,y,z}$ obtained by the gluing Lemma \ref{lem:condglu} of $\pi_{x,y}^\ast$ and $\pi_{y,z}^\ast$. Collecting all these bounds we have $\pi^\ast [\alpha(Z)d(X,Y)]\le  R[\alpha^q]\tilde W_{p,d}(P,Q)$. We obtain
$$
\frac{\pi_{x,z}^\ast[\alpha(Z)d(X,Z)]}{R[\alpha^q]^{1/q}}\le (\tilde W_{p,d}(P,Q)+\tilde W_{p,d}(Q,R)).
$$
and taking the supremum on $\alpha$ the desired result follows from the definition of $\tilde W_{p,d}(Q,R)$.
\end{proof}

\subsection{Markov couplings}
In this section, we only consider Markov couplings on the product space  $E^n$ with $n=2$, the cases $n\ge 2$ following by simple induction reasoning. \begin{dfn}
Let $P$, $Q$ $\in M(E^2)$, the set of Markov couplings $\tilde M(P,Q)$ are defined as the products $\pi=\pi_1\pi_{2|1}$ with $\pi_1$ a coupling of $P_1$ and $Q_1$ and $\pi_{2|1}$ a coupling of $P_{2|1}$ and $Q_{2|1}$.
\end{dfn}
The terminology of Markov couplings was introduced by R\"uschendorf in \cite{ruschendorf:1985}. Similar couplings are used by Marton in \cite{marton:1996a}. The property of conditional independence in the gluing Lemma \ref{lem:glu} is nicely compatible with Markov couplings:
\begin{lemma}\label{lem:condglu}
For any Markov couplings $\pi_{x,y}\in \tilde M(P,Q)$ and $\pi_{y,z}\in \tilde M(P,Q)$ with $P,\,Q,\,R\in \tilde M(E^2)$ it exists a distribution $\pi_{x,y,z}$ with corresponding margins  and such that $X=(X_1,X_2)$ and $Z=(Z_1,Z_2)$ are independent conditional on $Y=(Y_1,Y_2)$.
\end{lemma}
\begin{proof}
By assumption $\pi_{x,y}=\pi_{x_1,y_1}\pi_{x_2,y_2|x_1,y_1}$ and $\pi_{y,z}=\pi_{y_1,z_1}\pi_{y_2,z_2|y_1,z_1}$. Let us define $\pi_{x,y,z}$ as $\pi_{x_1,y_1,z_1}\pi_{x_2,y_2,z_2|x_1,y_1,z_1}$ by the relation
\begin{equation}\label{eq:glu}
\pi_{x_1,y_1,z_1} =\pi_{x_1|y_1} \pi_{ z_1|y_1} \pi_{y_1} ,
\end{equation}
and
\begin{equation}\label{eq:condglu}
\pi_{x_2,y_2,z_2|x_1,y_1,z_1} =\pi_{x_2|x_1,y_1,y_2} \pi_{ z_2|y_1,z_1,y_2} \pi_{y_2|y_1} .
\end{equation}
Let us check that $\pi_{x,y,z}$ has the correct margins. First, from the classical gluing lemma we know that $\pi_{x_1,y_1,z_1}$ has the correct margins. It remains to prove that $\pi_{x_2,y_2,z_2|x_1,y_1,z_1}$ has the correct margins. Notice that from the definition of  Markov couplings, we have $\pi_{y_2|y_1}=\pi_{y_2|x_1,y_1}=\pi_{y_2|y_1,z_1}$. Thus the first margin of $\pi_{x_2,y_2,z_2|x_1,y_1,z_1}$ is equal to
$$
\pi_{x_2|x_1,y_1,y_2}\pi_{y_2|y_1}=
\pi_{x_2|x_1,y_1,y_2}\pi_{y_2|x_1,y_1}=\pi_{x_2,y_2|x_1,y_1}.
$$
The same reasoning show that the second margin is also the correct one.\\

We proved above that by construction $X_1$ and $Z_1$ are independent conditional on $Y_1$, i.e. that $\pi_{x_1,z_1|y_1}=\pi_{x_1|y_1}\pi_{z_1|y_1}$. Let us show that it is also the case conditional on $Y_1$ and $Y_2$. We have
$$
\pi_{x_1,z_1|y_1,y_2}=\frac{\pi_{x_1,z_1,y_1,y_2}}{\pi_{y_1,y_2}}=\frac{\pi_{y_2|y_1}\pi_{x_1,z_1,y_1}}{\pi_{y_2|y_1}\pi_{y_1}}=\pi_{x_1,z_1|y_1}
$$
the third identity following from the identity $\pi_{y_2|y_1}=\pi_{y_2|x_1,y_1,z_1}$ by the identity \eqref{eq:condglu}. Thus, using that $X_1$ and $Z_1$ are independent conditional on   $Y_1$ we obtain the identity $\pi_{x_1,z_1|y_1,y_2}=\pi_{x_1|y_1}\pi_{z_1|y_1}$. We conclude that $\pi_{x_1,z_1|y_1,y_2}= \pi_{x_1 |y_1,y_2}\pi_{ z_1|y_1,y_2}$ as
$$
\pi_{x_1|y_1}=\frac{\pi_{y_2|y_1}\pi_{x_1 ,y_1}}{\pi_{y_2|y_1}\pi_{y_1}}=\frac{\pi_{x_1,y_1,y_2}}{\pi_{y_1,y_2}}=\pi_{x_1|y_1,y_2}
$$
the third identity following from the identity $\pi_{y_2|y_1}=\pi_{y_2|x_1,y_1}$ by definition of Markov couplings (the same is true replacing $x_1$ by $z_1$).\\

It remains to prove that $X_2$ is independent of $Z_2$ conditional on $(X_1,Z_1)$ and $(Y_1,Y_2)$. Indeed, we have by construction 
$$
\pi_{x_2,z_2|x_1,y_1,z_1,y_2}=\frac{\pi_{x_2,y_2,z_2|x_1,y_1,z_1}}{\pi_{y_2|x_1,y_1,z_1}}
=\frac{\pi_{x_2,y_2,z_2|x_1,y_1,z_1}}{\pi_{y_2|y_1}}=
\pi_{x_2|x_1,y_1,y_2}
\pi_{z_2|z_1,y_1,y_2},
$$
the last identity following from the identity  \eqref{eq:condglu}.
Thus the result is proved.\end{proof}

\subsection{Weak transport costs on $E^n$, $n\ge 2$}

We extend the  definition of $\tilde W$ on the product space $E^n$ for $n\ge 2$. Let $P$, $Q$ $\in M(E^n)$ we define
\begin{equation}
\tilde W_{p,d}(P,Q)=\sup_{\alpha\in M^+(E^n)}\inf_{\pi\in \tilde M(P,Q)}\frac{\sum_{j=1}^n\pi[\alpha_j(Y)d(X_j,Y_j)]}{(\sum_{j=1}^nQ[\alpha_j(Y)^q])^{1/q}}
\end{equation}
with the convention $(\sum_{j=1}^nQ[\alpha_j(Y)^q])^{1/q}=\max_{1\le j\le n}\mbox{ess}\sup \alpha_j$ if $q=\infty$ and
\begin{equation}\label{eq:walpha}
\tilde W_{\alpha,d}(P,Q)=\inf_{\pi\in \tilde M(P,Q)}{\sum_{j=1}^n\pi[\alpha_j(Y)d(X_j,Y_j)]}
\end{equation}
for any fixed $\alpha=(\alpha_j)_{1\le j\le n}\in M^+(E^n)$. Considering Markov couplings, we  use the conditional independence in the gluing Lemma \ref{lem:condglu} to assert that the weak transport cost on $E^n$ also satisfies the triangular inequality. More useful, $\tilde W_{\alpha,d}$ satisfies an inequality similar than the triangular one:
\begin{lemma}\label{lem:wal}
For any $P,Q,R \in M(E^n)$, for any $\alpha \in M^+(E^n)$ there exists $\tilde \alpha\in M^+(E^n)$ satisfying  $Q[\tilde \alpha_j(Y)]^q\le R[\alpha_j^q(Z)]$ for all $1\le j \le n$ and
\begin{equation}\label{eq:wal}
\tilde W_{\alpha,d}(P,R)\le \tilde W_{\tilde \alpha,d}(P,Q)+\tilde W_{\alpha,d}(Q,R)
\end{equation}
\end{lemma}
\begin{rmk} 
As a consequence of the Lemma \ref{lem:wal}, we obtain the triangular inequality for $\tilde W$
\begin{equation}\label{ti1}
\tilde W_{p,d}(P,R)\le \tilde W_{p,d}(P,Q)+\tilde W_{p,d}(Q,R)
\end{equation}
by taking the supremum on $\alpha$ on both sides of \eqref{eq:wal} and using the relation $Q[\tilde \alpha_j(Y)]^q\le R[\alpha_j^q(Z)]$.
\end{rmk}
\begin{proof}
Let us fix  $\alpha \in M^+(E^n)$ such that $R[\alpha_j^q]<\infty$ for all $1\le j\le n$. Define recursively the couplings  $\pi_{y,z}^\ast$ and $\pi_{x,y}^\ast$ $\in\tilde M(E^2)$ such that 
\begin{eqnarray*}
\pi_{y,z}^\ast\Big[\sum_{j=1}^n \alpha_j(Z)d(X_j,Z_j)\Big]& =&  \tilde W_{\alpha,d}(Q,R),\\
\pi_{x,y}^\ast\Big[\sum_{j=1}^n\pi_{z|y}^\ast[\alpha_j(Z)|Y]d(X_j,Y_j)\Big]&=&   \tilde W_{\pi_{z|y}^\ast[\alpha(Z)|Y],d}(P,Q).\end{eqnarray*}
where we use Jensen's inequality.
Let us denote $\pi^\ast=\pi^\ast_{x,y,z}$ obtained by the gluing Lemma \ref{lem:condglu} of $\pi_{x,y}^\ast$ and $\pi_{y,z}^\ast$.
Then 
\begin{eqnarray}
\nonumber \pi_{x,z}^\ast\Big[\sum_{j=1}^n\alpha_jd(X_j,Z_j)\Big]&\le&\pi_{y,z}^\ast\Big[\sum_{j=1}^n\alpha_j(Z)d(X_j,Y_j)\Big]+\pi^\ast\Big[\sum_{j=1}^n\alpha_j(Z)d(Y_j,Z_j)\Big]\\
\nonumber &\le&\pi^\ast_{x,y}\Big[\sum_{j=1}^n\pi^\ast_{z,y}[\alpha_j(Z)|Y]d(X_j,Y_j)\Big] \\
\nonumber &&+\pi^\ast_{y,z}\Big[\sum_{j=1}^n\alpha_j(Z)d(Y_j,Z_j)\Big]\\
\label{eq:w1}&\le& \tilde W_{\pi_{z|y}^\ast[\alpha(Z)|Y],d}(P,Q)+\tilde W_{\alpha,d}(Q,R).\end{eqnarray}
The inequality \eqref{eq:wal} follows from \eqref{eq:w1} taking $\tilde \alpha_j=\pi_{y,z}^\ast[\alpha_j(Z)|Y=\cdot]$ and noticing that   the relation $Q[\tilde \alpha_j^2(Y)]\le R[\alpha_j^2(Z)]$ holds by an application of Jensen's inequality.
\end{proof}

\section{Weak transport inequalities}\label{sec:wti}
\subsection{Weak transport inequalities and dual forms}
Let us say that the probability measure $P$ on $E^n$, $n\ge 1$, satisfies the weak transport inequality $\tilde T_{p,d}(C)$ when for all distribution $Q$ on $E^n$ we have
\begin{equation}\label{wtci}
\tilde W_{p,d}(P,Q)\le \sqrt{2C\mathcal K(Q|P)}.
\end{equation}
Let us say that $P$ satisfies the inverted weak transport inequality  $\tilde T_{p,d}^{(i)}(C)$ when 
\begin{equation}\label{iwtci}
\tilde W_{p,d}(Q,P)\le \sqrt{2C\mathcal K(Q|P)}.
\end{equation}
By an application of Jensen's inequality, $P$ satisfies $\tilde T_{p,d}(C)$ and $\tilde T_{p,d}^{(i)}(C)$ as soon as $\tilde T_{p',d}(C)$ and $\tilde T_{p',d}^{(i)}(C)$ reciprocally with $p'\ge p$. By the dual form of the $\ell^p$-norm $\|\cdot\|_p$, the weak transport inequalities $\tilde T_{p,d}^(C)$ or $\tilde T_{p,d}^{(i)}(C)$ on the product space $E^n$ where $E$ is equipped with $d$ are equivalent, respectively, to $\tilde T_{p,d}^(C)$ or $\tilde T_{p,d}^{(i)}(C)$ on $E^n$ equipped with $d^{\ell_p}p(x,y)=\|(d(x_j,y_j))_{1\le j\le p}\|_p$. Thus, when $d$ is not specified, we consider the case $n=1$ only with no loss of generality. We have $\tilde T_{1,d}(C)=\tilde T_{1,d}^{(i)}(C)= T_{1,d}(C)$ where $T_{1,d}(C)$ is the classical transport inequality  defined by the relation
$$
\inf_{\pi\in\tilde M(P,Q)} \pi[d(X_j,Y_j)]\le \sqrt{2C\mathcal K(Q|P)}.
$$
Following \cite{bobkov:gotze:1999}, we investigate the dual form of the weak transport. Denote $$f_{\alpha,d}(y)=\inf_{x\in E } \{ \alpha (y)d(x ,y )+f(x) \}$$ and $\mathcal C_b$ the set of continuous bounded functions with values in $\R$, we have the following dual forms of the weak transport inequalities:
\begin{thm}\label{th:E1}
The weak transport inequalities $\tilde T_{p,d}(C)$ and $\tilde T_{p,d}^{(i)}(C)$ are equivalent, respectively, to
\begin{align}\label{eq:df}
\sup_{\lambda>0}\sup_{\alpha\in M^+(E)}\sup_{f\in \mathcal C_b}P\Big[\exp\Big(\lambda (f_{\alpha,d} -P[f])- C\lambda^2\Big(\frac{\alpha^q-1}q+\frac12\Big)\Big)\Big]&\le 1,\\
\label{eq:df1}
\sup_{\lambda>0}\sup_{\alpha\in M^+(E)}\sup_{f\in \mathcal C_b}P\Big[\exp\Big(\lambda (f_{\alpha,d} -P[f])-C\lambda^2\Big(\frac{P[\alpha^q]-1}q+\frac12\Big)\Big)\Big]&\le 1.
\end{align} 
\end{thm}
\begin{proof}
As their proofs are similar, we prove the first dual form only.
From the dual form of $\tilde W_{\alpha,d}$ for $\alpha\in M^+(E)$ fixed we have
$$
\tilde W_{\alpha,d}(P,Q)=\inf_\pi \pi\Big[\alpha(Y)d(X_j,Y_j)\Big]=\sup_{f\in \mathcal C_b}Q[f_{\alpha,d}]-P[f].
$$
Then a measure $P$ satisfies  $\tilde T_{p,d}(C)$ if  for any $\alpha\in M^+(E)$ and any probability measure $Q$
$$
\sup_{f\in \mathcal C_b}Q[f_{\alpha,d}]-P[f]\le Q[\alpha^q]^{1/q}\sqrt{2C\mathcal K(Q|P)}.
$$
From the variational identity 
$$
ab=\inf_{\lambda>0}\lambda  a^q/q+b^p/(\lambda^{p-1} p)
$$
we get for all $\lambda>0$:
$$
Q[(f_{\alpha,d}-P[f])]\le \lambda C /q Q[\alpha^q]+ \mathcal K(Q|P)^{p/2} 2^{p/2}C^{1-p/2}/(\lambda^{p-1}p).
$$
We can rewrite it as
$$
(p/2)Q[(p/C)^{1-p/2}\lambda^{p-1}(f_{\alpha,d}-P[f]-\lambda C   /q \alpha^q)]^{2/p}\le \mathcal K(Q|P).
$$
From the Young inequality
$$
(p/2)x^{2/p}\ge yx -(1-p/2)y^{2/(2-p)} 
$$
applied with $y=(C\lambda^2/p)^{2/p-1}$ we obtain
$$
(p/2)((p/C)^{1-p/2}\lambda^{p-1})^{2/p}x^{2/p}\ge x-(1-p/2)C\lambda^2/p
$$
For $x=Q[\lambda(f_{\alpha,d}-P[f]-\lambda C   /q \alpha^q)]$ we obtain
$$
Q [\lambda(f_{\alpha,d}-P[f]- { \lambda C}  \alpha^q/q)]-\mathcal K(Q|P)\le (1/p-1/2) C\lambda^{2}.
$$
Then the desired result follows from the variational formula of the entropy.
 \end{proof}
In the case $p=1$, then $q=\infty$ and the dual forms \eqref{eq:df} and \eqref{eq:df1} only depends on $\alpha$ through the fact that $\alpha(y)\le 1$, $y\in E$. Then one can consider $\alpha=1$, $f_{\alpha,d}(y)=\inf_{x\in E}\{d(x,y)+f(x)\Big\}$ that forces to consider Lipschitz functions and we recognize the dual form of the  transport inequality $T_{1,d}$ that is the Hoeffding inequality:
$$
\sup_{\lambda>0}\sup_{f \in{\tiny \rm Lip}_1(d)}P\Big[\exp\Big(\lambda (f  -P[f])- \frac{C\lambda^2}2\Big)\Big]\le 1.
$$
Here $\Lip_1$ is the set of $1$-Lipschitz functions $f$ with respect to $d$.

To obtain similar results when $p=2$, it is crucial to identify the map $f\to f_{\alpha,d}$. In the sequel, we focus on the cases $d$ the Hamming distance and the Euclidian norm in $\R^n$.

\subsection{The specific case $d=1$, the Hamming distance}\label{sec:hd} The weak transport inequalities when $d(x,y)=1_{x\neq y}$, denoted $d=1$, was introduced in \cite{marton:1996a}. When $n=1$ we have the explicit expression $f_{\alpha,1}(y)= (\alpha(y)+\inf f)\wedge f(y)$. As the difference $f_{\alpha,1}-f$ is unchanged when adding a constant on $f$, we can take $\inf f=0$ with no loss of generality and
$$
\sup_{\alpha>0}P[\exp(\lambda (f_{\alpha,1} -P[f])-(\lambda \alpha)^2C/2)]=P[\exp(\lambda(f-P[f])-\lambda^2f^2C/2))].
$$
But for any $X>0$ we have $X-X^2/2\le \log(1+X)$ and thus
$$
P[\exp(X-X^2/2)]\le 1+P[X]\le \exp(P[X]).
$$
$\tilde T_{2,1}(1)$ follows by taking $X=\lambda f$.
For the inverted weak transport, we apply the inequality 
$$P[\exp(-X)]\le 1-P[X]+P[X^2]/2\le \exp(-P[X]+P[X^2]/2)$$ and we obtain
\begin{cor}\label{cor:hd}
Any measure $P$ on $E$ satisfies $\tilde T_{2,1}(1)$ and $\tilde T_{2,1}^{(i)}(1)$.
\end{cor}
\begin{rmk}
Alternative proofs of the Corollary \ref{cor:hd} are given in \cite{marton:1996a} and a stronger version is given in \cite{samson:2000}. \end{rmk}Consider its extension $\tilde T_{2,1_2}$ on $E^n$, where $1_2$ denotes
$$
1_2(x,y)^2=\sum_{j=1}^n1_{x_j\neq y_j},\qquad x,y\in E^n.
$$
Extending the previous reasoning, we obtain the dual form
\begin{prop}\label{prop:d1}
The weak transport inequalities $\tilde T_{p,1_2}(C)$ and $\tilde T_{p,1_2}^{(i)}(C)$ are equivalent to that for any $f$ and $\alpha_j$, $1\le j\le n$, satisfying
\begin{equation}\label{eq:sb}
f(x)-f(y)\le \sum_{j=1}^n\alpha_j(x)1_{x_j\neq y_j},\qquad x,y\in E^n,
\end{equation}
it holds, respectively,
\begin{align*} 
P\Big[\exp\Big(\lambda (f-P[f])- \frac{C\lambda^2}2\sum_{j=1}^n\alpha_j^2\Big)\Big)\Big]&\le 1,\qquad \lambda>0,\\
P\Big[\exp\Big(\lambda (P[f]-f)-\frac{C\lambda^2}2\sum_{j=1}^nP[\alpha_j^2]\Big)\Big)\Big]&\le 1,\qquad \lambda>0.
\end{align*} 
\end{prop}
The weak transport inequalities $\tilde T_{2,1_2}(C)$ and $\tilde T_{2,1_2}^{(i)}(C)$, extended to $E^n$, $n>1$, are well suited to assert the concentration of measures, see \cite{marton:1996a}, especially through self bounding functions $f$, see \cite{boucheron:lugosi:massart:2013}. More generally, assume that $f$ satisfies \eqref{eq:sb} with functions $\alpha_j$ such that
$
\sum_{j=1}^n\alpha_j^2\le f$.
Let $A\subset E^n$ be a measurable set. Denoting
$$
d_T(x,A)=\sup_{\|c\|\le 1}\inf_{y\in A}\sum_{j=1}^nc_j1_{x_j\neq y_j},
$$
noticing that $d_T^2(x,A)/4$ is self-bounding, we derive the Talagrand inequality \cite{talagrand:1995} in its generalized form given in \cite{johnson:schechtman:1991}:
\begin{prop}\label{pr:tal}
If the law $P$ of $X=(X_1,\ldots,X_n)$ satisfies $\tilde T_{2,1_2}(C)$ and $\tilde T_{2,1_2}^{(i)}(C)$ then 
$$
P\Big[\exp\Big(\frac{d_T^2(x,A)}{4C}\Big)\Big]\le \frac1{P(A)}.
$$
\end{prop}
\begin{rmk}
If the $X_i$ are independent, Theorem \ref{th:wti} yields that $P$ satisfies $\tilde T_{2,1_2}(1)$ or $\tilde T^{(i)}_{2,1_2}(1)$ and the constant $C=1$ is optimal, see \cite{talagrand:1995}.\end{rmk}
\begin{proof}
Let $c^\ast$ be the weights that achieves the supremum in $d_T$. We have 
\begin{align*}
d_T(x,A)-d_T(y,A)&\le \inf_{x'\in A}\sum_{j=1}^nc_j^\ast(x)1_{x_j\neq x'_j}-\inf_{y'\in A}\sum_{j=1}^nc_j^\ast(y)1_{y_j\neq y'_j}\\
&\le \sum_{j=1}^nc_j^\ast(x)1_{x_j\neq y_j}.
\end{align*}
Then, by the convex inequality $x^2-y^2\le 2x(x-y)$, we obtain
$$
d_T(x,A)^2-d_T(y,A)^2\le  \sum_{j=1}^n 2d_T(x,A)c_j^\ast(x)1_{x_j\neq y_j}.
$$
Thus $f(x)=d_T(x,A)^2$ satisfies \eqref{eq:sb} with $\alpha_j(x)$ satisfying
$
\sum_{j=1}^n\alpha_j(x)\le 4d_T(x,A)^2
$. Applying the first inequality in Proposition \ref{prop:d1} with $\lambda=1/(2C)$ we obtain 
$P[\exp(d_T^2(X,A)/(4C))]\le\exp(P[d_T^2(X,A)]/(2C)]).$

Using directly \eqref{iwtci}, we also have
$$
P[d_T^2(X,A)]\le Q[d_T^2(X,A)]+\sqrt{2CP[d_T^2(X,A)]\mathcal K(Q|P)},\qquad Q\in M^+(E^n).
$$
Choosing $Q$ as $P$ restricted to $\{d_T(X,A)=0\}=\{X\in A\}$, noticing that $Q[d_T^2(X,A)]=0$ and $\mathcal K(Q|P)=-\log(P(A))$, we obtain 
$
P[d_T^2(X,A)]\le \sqrt{-2C\log(P(A))}$, equivalent to $\exp(P[d_T^2(X,A)]/(2C)])\le 1/P(A)$. The desired result follows.
\end{proof}

\subsection{The specific case $d=N$, the euclidian metric}
Next we consider the case of $E=\R^n$ equipped with the euclidian norm $\|\cdot\|$ that we denote $d=N$. We obtain
\begin{thm}\label{th:tsi}
The weak transport inequalities $T_{2,N}(C)$ and $T_{2,N}^{(i)}(C)$ are equivalent, respectively, to
\begin{align}\label{eq:sg}
P[\exp(g-P[g]-C\|\nabla g\|^2/2]\le1,\mbox{ for any separately convex function $g$},\\
\label{eq:sg2}
P[\exp(g-P[g]-CP[\|\nabla g\|^2]/2]\le1,\mbox{ for any separately concave function $g$}.
\end{align}
\end{thm}
\begin{rmk}
The inequality \eqref{eq:sg} is called the Tsirels'on inequality who discovered it for independent Gaussian random variables with the optimal constant $C=1$. Corollary 6.1 in \cite{bobkov:gentil:ledoux:2001} states that it holds for any measures satisfying the log-Sobolev inequality. In particular, $\tilde T_{2,1_2}(C)$ holds for any log-concave measure $dP/dx=e^{-V}$ with $C$-strongly convex function $V$.
\end{rmk}
\begin{proof}
We consider only the case $n=1$ as the extension to $n\ge 1$ follows the same reasoning. First, we note that one can let $\alpha$ take non positive values in \eqref{eq:df} and \eqref{eq:df1}. A simple computation of the minimizer in the definition of $f_{\alpha,d}$ for sufficiently smooth $f$ gives the following identities: $f_{\alpha,d}(y)=f'(x)(y-x)+f(x)$ for $x$ satisfying $f'(x)^2=\alpha(y)^2$ or $f_{\alpha,d}(y)= f(y)$ if such $x$ does not exists. Thus, one can restrict us to the cases where $f'(x)(y-x)+f(x)\le f(y)$ and $f'(x)^2=\alpha(y)^2$. As $\alpha$ is a free parameter, it is also the case for $f'(x)$ and thus $x$. Thus we can restrict ourselves to convex function $f$ and we distinguish two cases: either $x=y$ or not. If $x$ varies, noticing that the dual form of the weak transport
inequality \eqref{eq:df} only depends on $x$ through 
$$
\lambda f_{\alpha,d}(y)-(\lambda \alpha(y))^2C/2=\lambda f'(x)(y-x)+\lambda f(x)-(\lambda f'(x))^2 C/2
$$
we derive by $x$ the function $g=\lambda f$ and we obtain  
$$
g''(x)(y-x-Cg'(x))=0.
$$
As $g$ is convex, the solution $Cg'(x)=y-x$ is excluded and thus worst $g$ are affine functions $g(x)=ax+b$ for some $a$, $b\in \R$. We obtain the condition
$$
P[\exp(a(X-P[X])-C a^2/2) ]\le 1,\qquad a\in \R,
$$
implied by \eqref{eq:sg}.
If $x=y$, then we also obtain
\eqref{eq:sg}. In any case \eqref{eq:sg} is a necessary and sufficient condition.

Now let us prove the equivalence for $T_{1,d}^{(i)}(C)$. We first notice that, when $n=1$, the dual form \eqref{eq:df1} is equivalent to
$$
\sup_{\lambda>0}\sup_{\alpha\in M^+(E)}\sup_{g\in \mathcal C_b}P[\exp(\lambda g  -P[\lambda g_{\alpha} +C\lambda^2 \alpha^2/2])] \le 1
$$
where $g_\alpha(y) =\sup_x \{g(x)-\alpha(y)|x-y|\}$. Following the same reasoning than above, one can consider only the cases where $g(x)-g'(x)(x-y)\ge g(y)$ and $g'(x)^2=\alpha(y)^2$. Denoting $\lambda g=f$, if one can let $x$ vary, we optimize the term 
$$
f(x)-f'(x)(x-y)+Cf'(x)^2/2
$$
by taking $f$ an affine function. Otherwise, $x=y$ and we obtain that any concave function $g$ satisfies
\eqref{eq:sg2}.
\end{proof}
We  relate weak transport inequalities to more classical notions of concentration. Recall that a measure $P$ on $E=\R^n$ is sub-gaussian if there exists $c>0$ such that
$$
P[\exp(\lambda \|X\|^2)]<\infty\quad\mbox{for}\quad0<\lambda<c.
$$
This property is equivalent to $T_{1,N}(C)$ for some $C>0$, see \cite{djellout:guillin:wu:2004,bolley:villani:2005}, and it is a very common assumption in statistics. We say that $P$ satisfies the {\it convex} Poincar\'e inequality if for any separately convex function $g$
$$
P[(g-P[g])^2]\le C P[\|\nabla g\|^2].
$$
\begin{rmk}
The convex Poincar\'e inequality on $E=\R$ has been studied in \cite{bobkov:gotze:1999a}. It is satisfied for $X$ standard normal or in $[0,1]$ with $C=1$. It also holds with the same constant $C$ for the product measure on $\R^n$, $n>1$. 
\end{rmk}
Notice that the convex Poincar\'e inequality is equivalent to a {\it concave} Poincar\'e inequality. We obtain
\begin{thm}
The weak transport inequality $T_{2,N}$ or $T_{2,N}^{(i)}$ implies sub-gaussianity and convex Poincar\'e inequality.\end{thm}
\begin{rmk}
In a personal communication, N. Gozlan and P.-M. Samson showed me that the converse is not true using the counter-example given p.15 in \cite{gozlan:2012}.\end{rmk}
\begin{proof}
The arguments developed in this proof are classical, see \cite{ledoux:2005}. We detail the case of $T_{2,N}$ when $n=1$ because the proof for $n>1$ and $T_{2,N}^{(i)}$ follows the same reasoning.
Assume that $P$ satisfies $T_{2,d}(C)$ or $T_{2,d}^{(i)}(C)$ and apply \eqref{eq:sg}  to $g(x)=\lambda x$: $P[\exp(\lambda(X-P[X]))]\le \exp(C\lambda^2/2)$, $\lambda>0$. Then $P$ must be sub-gaussian. Now, applying \eqref{eq:sg} or \eqref{eq:sg2} to $tg$ with $t\to 0$ we obtain the convex Poincar\'e inequality in both cases.
\end{proof}
Tsirel'son inequality \eqref{eq:sg} quantifies the concentration of "self bounding" functions with respect to the euclidian norm, i.e. convex functions $f$ such that $\|\nabla f\|^2\le f$. Let $A$ be a measurable set of $\R^n$ and $B$ its convex hull, then $d_N(x,B)=\inf_{y\in B}\|x-y\|/4$ is a self bounding function. Following the same reasoning than in the proof of Proposition \ref{pr:tal}, we obtain the Euclidian version of Talagrand's concentration inequality \cite{maurey:1991}
\begin{prop}\label{pr:tal2}
If the law $P$ of $X=(X_1,\ldots,X_n)$ satisfies $\tilde T_{2,N}(C)$ and $\tilde T_{2,N}^{(i)}(C)$ then
$$
P\Big[\exp\Big(\frac{d_N^2(X,B)}{4C}\Big)\Big]\le \frac1{P(A)}.
$$
\end{prop}
\begin{rmk}
The result is proved for independent $X_j$s on $[0,1]$ or standard normal with the optimal constant $C=1$ in \cite{maurey:1991} via the convex  property ($\tau$). \end{rmk}

\subsection{Coupling trajectories}
Dual forms provided in Theorem \ref{th:E1} are particularly powerful to derive weak transport inequalities when $n=1$. 
In order to obtain concentration for measures on $E^n$, $n>1$, we prove that the weak transport inequalities hold for non product measures. To obtain constants as sharp as possible, we couple trajectories via the new notion of  $  \Gamma_{d,d'}(p)$-weak dependence.

To any law $P$ on $E^n$, add
artificially time $0$ and put $X_0 = Y_0 = x_0 =y_0$ for a fixed point $y_0\in E$. Denote $x^{(i)}=(x_i,\ldots,x_0)$ for $i\ge 0$ and $P_{|x^{(i)}}$ the conditional laws of $(X_{i+1},\ldots,X_n)$ given that $(X_i,\ldots,X_0)=x^{(i)}=(x_i,\ldots,x_0)$. Let $d$ and $d'$ be two lower semi-continuous distances on $E$ such that $d\le Md'$ for some $M>0$. 
Let us work under the following weak dependence assumption: 
\begin{dfn}
For any $1\le  p\le 2$, the probability measure $P$ is $\Gamma_{d,d'}(p)$-weakly dependent  if for any $1\le i\le n$, any $(x^{(i)},y_i)\in E^{i+2}$ there exists a coupling scheme $\pi_{|i}$ of $(P_{|x^{(i)}},P_{|x^{(i-1)},y_i})$ and coefficients $\gamma_{k,i}(p)\ge0$ such that 
\begin{equation}\label{eq:wd}
W_{p,d}(P_{x_k|x^{(i)}},P_{x_k|x^{(i-1)},y_i})\le \gamma_{k,i}(p)\,d'(x_i,y_i), \qquad \forall i<k\le n.
\end{equation}
\end{dfn}
Let us denote 
$$
\Gamma(p)=\begin{pmatrix}M&0&0&\ldots&0\\
\gamma_{2,1}(p)&M&0&\ldots&0\\
\gamma_{3,1}(p)&\gamma_{3,2}(p)&M&\ddots&\vdots\\
\vdots&\vdots&\ddots&1&0\\
\gamma_{n,1}(p)&\gamma_{n,2}(p)&\ldots&\gamma_{n,n-1}(p)&M
\end{pmatrix}.
$$
The matrix $
\Gamma(p)$ has $n$ rows and $n$ columns. We equip $\R^n$ with the $\ell^p$-norm and the set of the matrix of size $n\times n$ with the subordinated norm, both denoted $\|\cdot\|_p$ for any $1\le p\le \infty$.
\begin{thm}\label{th:wti}
For any $1\le p\le 2$, if $P$ is $\Gamma_{d,d'}(p)$-weakly dependent and $P_{x_j|x^{(j-1)}}$ satisfies $\tilde T_{p,d}(C)$ or $ \tilde T_{p,d}^{(i)}(C)$ for all $1\le j\le n$ then  $P$ satisfies  respectively,  $\tilde T_{p,d_p}(C\|\Gamma(p)\|_p^2n^{2/p-1})$ or  $\tilde T^{(i)}_{p,d_p}(C\|\Gamma(p)\|_p^2n^{2/p-1})$,.\end{thm}
\begin{rmk}\label{rk:st}When the process $(X_t)$ is stationary, we have $\gamma_{i,j}(p)=\gamma_{k,\ell}(p)$ for $j-i=k-\ell$. From the basic inequality $\|A\|_p\le \|A\|_1^{1/p}\|A\|_\infty^{1-1/p}$ and the fact that $\|\Gamma\|_1=\|\Gamma\|_\infty=M+\sum_{i=1}^n\gamma_{i,0}(p)$, then $\|A\|_p\le M+\sum_{i=1}^n\gamma_{i,0}(p)$.
\end{rmk}
\begin{proof}
The proofs of the two assertions are similar as the weak dependence condition \eqref{eq:wd} is symmetric in $x_i$ and $y_i$. Thus the proof of the second assertion is omitted.

Let us fix  $\alpha \in M^+(E^n)$ such that $Q[\alpha_j^q]<\infty$ for all $1\le j\le n$. As  preliminaries, we recall the following result of existence of the optimal Markov coupling due to from R\"uschendorf, \cite{ruschendorf:1985} and a simple and useful consequence of this result stated in Lemma \ref{prop:tens}.

Let $\sigma:E^n\times E^n\mapsto \R_+$ and the section of $\sigma$ in $(x_1,y_1)\in E^2$ as
$$
\sigma_{x_1,y_1}(x_2,y_2)=\sigma ((x_1,x_2),(y_1,y_2)).
$$
\begin{thm}[Theorem 3 in  \cite{ruschendorf:1985}]\label{th:ru}
We have the equivalence between  \begin{enumerate}
\item $\inf_{\pi\in \tilde M}\pi[\sigma]= \pi^\ast[\sigma]$ with $\pi^\ast \in \tilde M$,
\item \begin{enumerate}
\item $h(x,y):=\inf_{\pi_{2|1}}\pi[\sigma_{x,y}]= \pi^\ast_{2|1}[\sigma_{x,y}|(x,y)]$ is finite $\pi_1-$a.s. and
\item $\inf_{\pi_1}\pi_1[h]=\pi_1^\ast[h]<\infty$.
\end{enumerate}
\end{enumerate}
\end{thm}
Denote $\alpha_j^{(i)}$ denotes the section of $\alpha_j$ in $y^{(i)}$ as 
$
\alpha_j^{(i)}(y_{i+1},\ldots,y_n)=\alpha_j(y)
$ and $\alpha^{(i)}=(\alpha_j^{(i)})_{j>i}$.
A simple corollary of this Theorem is the following result:
\begin{lemma}\label{prop:tens}
Let $P$, $Q\in M(E^n)$ be decomposed as $P=P_{1}P_{|X_1}$ and $Q=Q_{1}Q_{|Y_1}$ for $P_{ 1}$, $Q_{ 1}\in M(E)$ and $P_{ |x_1}$, $Q_{|y_1}\in M(E^{n-1})$. Then for any $\alpha\in M^+(E^n)$ and any coupling $\pi_1\in \tilde M(P_1,Q_1)$ we have
\begin{equation}\label{eq:tens}
\tilde W_{\alpha,d}(P,Q)\le   \pi_1[Q_{|Y_1}[\alpha_1|Y_1]d(X_1,Y_1)+\tilde W_{\alpha^{(1)},d}(P_{|X_1},Q_{|Y_1})].
\end{equation}
\end{lemma}
\begin{proof}
Let us assume that for almost all $x_1$, $y_1\in E$ we have $\tilde W_{\alpha^{(1)},d}(P_{|x_1},Q_{|y_1})<\infty$. Then, by lower semi-continuity, it exists $\pi^\ast_{|x_1,y_1}$   such that:
$$
\pi^\ast_{|x_1,y_1}\Big[\sum_{j=2}^n\alpha_j^{(1)}d(X_j,Y_j)\Big]=\tilde W_{\alpha^{(1)},d}(P_{ |x_1},Q_{ |y_1})].
$$
Thus the desired result follows from Theorem \ref{th:ru} remarking that for any $x_1$, $y_1\in E$ we have 
$$\pi^\ast_{|x_1,y_1}[\alpha_1^{(1)}d(x_1,y_1)]=\pi^\ast_{|x_1,y_1}[\alpha_1^{(1)}|x_1,y_1]d(x_1,y_1)=Q_{|y_1}[\alpha_1|y_1]d(x_1,y_1)$$ by definition of Markov couplings.
\end{proof}
Let us consider now the following coupling scheme denoted $\tilde \pi$  defined recursively as $\tilde \pi=\tilde \pi_{n|n-1}\cdots\tilde \pi_{2|1}\tilde \pi_{1|0}$ $\in\tilde M(E^n)$ where $\tilde \pi_{j|j-1}=\tilde \pi_{x_j,y_j|x^{(j-1)},y^{(j-1)}}$  is determined such that
\begin{multline}\label{eq:tpi}
\tilde \pi_{j|j-1}\Big[\sum_{k=j}^nQ_{|Y_j,y^{(j-1)}}[\alpha_k^q|Y_j,y^{(j-1)}]^{1/q}\gamma_{k,j}(p)d'(X_j,Y_j)\Big]  \\= \Big(\sum_{k=j}^nQ_{|y^{(j-1)}}[\alpha_k^q| y^{(j-1)}]\gamma_{k,j}(p)^q\Big)^{1/q}
\tilde W_{p,d'}(P_{x_j|x^{(j-1)}},Q_{y_j|y^{(j-1)}})
\end{multline}
for all $x^{(i-1)},y^{(i-1)}$ in $E^{i-1}$.

We are now ready to prove the result iterating several times the same reasoning. Let us detail the case $j=1$ when considering probabilities conditional on $y_0$. Applying \eqref{eq:tens} and \eqref{eq:wal} we have
\begin{eqnarray}
\nonumber \tilde W_{\alpha,d}(P,Q)&\le& \tilde \pi_{1|y^{(0)}}[Q_{|Y_1,y^{(0)}}[\alpha_1|Y_1,y^{(0)}]d(X_1,Y_1)+\tilde W_{\alpha^{(1)},d}(P_{ |X_1,y^{(0)}},Q_{ |Y_1,y^{(0)}})]\\
\label{cb1}&\le& \tilde \pi_{1|y^{(0)}}[Q_{|Y_1,y^{(0)}}[\alpha_1|Y_1,y^{(0)}]d(X_1,Y_1)+\tilde W_{\alpha^{(1)},d}(P_{|Y_1,y^{(0)}},Q_{|Y_1,y^{(0)}})\\
\nonumber &&+\tilde W_{\tilde \alpha^{(1)},d}(P_{|X_1,y^{(0)}},P_{|Y_1,y^{(0)}})].
\end{eqnarray}
To bound the last term, we  use the definition of the $\Gamma_{d,d'}(p)$-weak dependence:
\begin{lemma}
For any $\alpha_k\in M^+(E)$ for all $j<k\le n$ and any $\Gamma_{d,d'}(p)$-weakly dependent probability measure $P$ we have 
\begin{equation}\label{eq:wdi}
\tilde W_{\alpha^{(j)},d}(P_{|x_j,y^{(j-1)}},P_{|y^{(j)}})\le \sum_{k=j+1}^nQ_{|y^{(j)}}[\alpha^q_k|y^{(j)}]^{1/q}\gamma_{k,j}(p)d'(x_j,y_j)
\end{equation}
\end{lemma}
\begin{proof}
Assume that $Q[\alpha^q_k]<\infty$ for $j<k\le n$. Then, applying the Holder inequality and the definition of Markov couplings, we have
\begin{eqnarray*}
\tilde W_{\alpha^{(j)},d}(P_{|x_j,y^{(j-1)}},P_{|y^{(j)}})&=&\inf_{\pi_{|j}}\pi_{|j}\Big[\sum_{k=j+1}^n\alpha_kd(X_k,Y_k)\Big]\\
&\le& \inf_{\pi_{|j}} \sum_{k=j+1}^n\pi_{|j}[\alpha_k^q|y^{(j)}]^{1/q}\pi_{|j}[d^p(X_k,Y_k)]^{1/p}\\
&\le&\inf_{\pi_{|j}}  \sum_{k=j+1}^nQ_{|y^{(j)}}[\alpha_k^q|y^{(j)}]^{1/q} \pi_{|j}[d^p(X_k,Y_k)]^{1/p}\\
&\le& \sum_{k=j+1}^nQ_{|y^{(j)}}[\alpha_k^q|y^{(j)}]^{1/q}\gamma_{k,j}(p)d'(x_j,y_j)
\end{eqnarray*}
because the $ \Gamma(p)$-weak dependence condition ensures the existence of a coupling scheme satisfying
$$
 \pi_{|j}[d^p(X_k,Y_k)]^{1/p}\le \gamma_{k,j}(p)d'(x_j,y_j)\qquad \forall\, j<k\le n.
$$
\end{proof}
Denoting $\gamma_{i,i}=M$ for all $1\le i\le n$, note that by assumption we have  the relation$d(X_i,Y_i)\le \gamma_{i,i}d'(X_i,Y_i)$.
Collecting the bounds \eqref{cb1} and \eqref{eq:wdi}  we obtain
\begin{multline*}
\tilde W_{\alpha,d}(P_{|y^{(0)}},Q_{|y^{(0)}})\le \tilde \pi_{1|y^{(0)}}\Big[ \sum_{k= 1}^nQ_{|Y^{(1)},y_0}[\alpha^q_k|Y^{(1)},y_0]^{1/q}\gamma_{k,1}d'(X_1,Y_1)\\
+\tilde W_{\alpha^{(1)},d}(P_{|Y_1,y^{(0)}},Q_{|Y_1,y^{(0)}})].
\end{multline*}
Let us do the same reasoning than above for any $1\le j\le n$ conditional on $y^{(j)}$  on $\tilde W_{\alpha^{(j-1)},d}(P_{|y^{(j)}},Q_{|y^{(j)}})$. For any $1\le j\le n$, we obtain:
\begin{multline*}
\tilde W_{\alpha^{(j-1)},d}(P_{|y^{(j-1)}},Q_{|y^{(j-1)}})\le \tilde \pi_{j|y^{(j-1)}}\Big[ \sum_{k= j}^nQ[\alpha^q_k|Y_j,y^{(j-1)}]^{1/q}\gamma_{k,j}(p)d'(X_j,Y_j)\\
+\tilde W_{\alpha^{(j)},d}(P_{|Y_j,y^{(j-1)}},Q_{|Y_j,y^{(j-1)}})].
\end{multline*}
For the specific Markov coupling  considered here, the identity \eqref{eq:tpi} holds and
\begin{align*}
\tilde W_{\alpha^{(j-1)},d}&(P_{|y^{(j-1)}},Q_{|y^{(j-1)}})\\
\le & \Big(\sum_{k=j}^nQ_{| y^{(j-1)}}[\alpha_k^q| y^{(j-1)}]\gamma_{k,j}(p)^q\Big)^{1/q}
\tilde W_{p,d'}(P_{x_j|y^{(j-1)}},Q_{y_j|y^{(j-1)}})\\
&+\tilde \pi_{j|y^{(j-1)}}[\tilde W_{\alpha^{(j)},d}(P_{|Y_j,y^{(j-1)}},Q_{|Y_j,y^{(j-1)}})]\\
\le &  \sum_{k=j}^nQ_{| y^{(j-1)}}[\alpha_k^q| y^{(j-1)}]^{1/q}\gamma_{k,j}(p) 
\tilde W_{p,d'}(P_{ x_j|y^{(j-1)}},Q_{y_j|y^{(j-1)}})\\
&+\tilde \pi_{j|y^{(j-1)}}[\tilde W_{\alpha^{(j)},d'}(P_{|Y_j,y^{(j-1)}},Q_{|Y_j,y^{(j-1)}})]
\end{align*}
where the last inequality follows from  the concavity of $x\to x^{1/q}$ and Jensen's inequality.
Applying an inductive argument, we obtain
\begin{eqnarray*}
\tilde W_{\alpha,d}(P,Q)&\le&  Q\Big[\sum_{j=1}^n\sum_{k=j}^nQ[\alpha_k^q| Y^{(j-1)}]^{1/q}\gamma_{k,j}(p)
\tilde W_{p,d'}(P_{x_j|Y^{(j-1)}},Q_{y_j|Y^{(j-1)}})\Big]\\
&\le &  \sum_{j=1}^n\sum_{k=j}^nQ[\alpha_k^q ]^{1/q}\gamma_{k,j}(p)
Q[\tilde W_{p,d'}(P_{x_j|Y^{(j-1)}},Q_{y_j|Y^{(j-1)}})^p]^{1/p}\\
&\le &  \sum_{j=1}^n\sum_{k=j}^nQ[\alpha_k^q ]^{1/q}\gamma_{k,j}(p)
Q[2C\mathcal K(Q_{y_j|Y^{(j-1)}}|P_{x_j|Y^{(j-1)}})^{p/2}]^{1/p}
\end{eqnarray*}
the second inequality following from H\"older's and Jensen's inequalities and the last one from the assumption $P_{x_j|y^{(j-1)}}\in\tilde T_{p,d}(C)$.
Let us denote ${\bf Q}$ the row vector $(Q[\alpha_k^q]^{1/q})_{1\le k\le n}$ and ${\bf W}$ the column vector $(Q[2C\mathcal K(P_{x_j|Y^{(j-1)}}|Q_{y_j|Y^{(j-1)}})^{p/2}]^{1/p})'_{1\le j\le n}$.
With $<;>$ denoting the scalar product, we obtain
$$
\tilde W_{\alpha,d}(P,Q)\le <{\bf Q};\Gamma(p) {\bf W}>\le \|{\bf Q}\|_{q}\|\Gamma(p)\|_p\|{\bf W}\|_p.
$$
Note that we have the identities
\begin{eqnarray*}
\|{\bf Q}\|_{q}&=&\Big(\sum_{j=1}^nQ[\alpha_k^q ]\Big)^{1/q},\\
\|{\bf W}\|_{p}&=&\Big(\sum_{j=1}^nQ[2C\mathcal K(Q_{y_j|Y^{(j-1)}}|P_{x_j|Y^{(j-1)}})^{p/2}]\Big)^{1/p},\\ 
\mathcal K(Q|P)&=&\sum_{j=1}^nQ[2C\mathcal K(Q_{y_j|Y^{(j-1)}}|P_{x_j|Y^{(j-1)}})].\end{eqnarray*}
As $p/2\le 1$, successive applications of Jensen's   and H\"older's inequalities yield
\begin{eqnarray*}
\|{\bf W}\|_{p}&\le& \Big(\sum_{j=1}^nQ[2C\mathcal K(Q_{y_j|Y^{(j-1)}}|P_{x_j|Y^{(j-1)}})]^{p/2}\Big)^{1/p}\\
&\le&n^{1/p-1/2}\Big(\sum_{j=1}^nQ[2C\mathcal K(Q_{y_j|Y^{(j-1)}}|P_{x_j|Y^{(j-1)}})] \Big)^{1/2}\\
&\le&\sqrt{n^{2/p-1}2C\mathcal K(Q|P)}.\end{eqnarray*}
Finally, we obtain
$$
\frac{\sum_{j=1}^n\tilde \pi [\alpha_j(Y)d(X_j,Y_j)]}{(\sum_{j=1}^nQ[\alpha_j^q ])^{1/q}}\le \sqrt{2C\|\Gamma(p)\|_p^2 n^{2/p-1}\mathcal K(Q|P)}.
$$
The desired result follows by taking the supremum over all   $\alpha\in M^+(E^n)$.
\end{proof}

\section{Examples of $\Gamma_{d,d'}(p)$-weakly dependent processes}\label{sec:ex}
\subsection{$\Gamma_{d,d'}(1)$-weakly dependent examples}
When $p=1$, the dual form of $\tilde T_{1,d_1}(Cn)=\tilde T^{i}_{1,d_1}(Cn)=T_{1,d_1}(Cn)$ is the Hoeffding inequality (which is not dimension free). We then recover concentration results that have been proved using the bounded difference approach of \cite{mcdiarmid:1989}.  Applying   the Kantorovitch-Rubinstein inequality, we obtain an explicit expression of $\gamma_{k,i}(1)$:
\begin{multline*}
\sum_{k=i+1}^n \gamma_{k,i}(1)\\
=
\sup_{f\in{\tiny \rm Lip}_1(d_1)}\sup_{x^{(i)}, y_i}\frac{P[f(X_{i+1},\ldots,X_n)|x^{(i)}]-P[f(X_{i+1},\ldots,X_n)|y_i,x^{(i-1)}]|}{d'(x_i,y_i)}.
\end{multline*}
In the bounded case $d\le M$ and $d'=1$, the $\Gamma_{d,1}$(1)-weak dependence condition coincides with the one introduces in Rio \cite{rio:2000}. As the conditional probabilities $P_{x_j|x^{(j-1)}}$ automatically satisfy Pinsker's inequality $\tilde T_{1,1}(1/4)$, Theorem \ref{th:wti} recovers the Hoeffding inequality of \cite{rio:2000}.
The context of $\Gamma_{N,1}(1)$-weak dependence is extensive and we refer the reader to Section 7 of \cite{dedecker:prieur:2005} for a detailed study of many examples in this case, including causal functions of stationary sequences, iterated random 
functions, Markov kernels and expanding maps.

When $d=d'$ the $\Gamma_{d,d}$-weak dependence condition coincides with the condition $(C_1)'$  of \cite{djellout:guillin:wu:2004}: for any $f\in{\rm Lip}_1(d_1)$ it holds
$$
|P[f(X_{k+1},\ldots,X_n)|x^{(k)}]-P[f(X_{k+1},\ldots,X_n)|y_k,x^{(k-1)}]|\le Sd(x_k,y_k).
$$
From Remark \ref{rk:st} we have  $\|\Gamma(1)\|_1\le 1+S$ and thus Theorem \ref{th:wti} recovers the Hoeffding inequality of \cite{djellout:guillin:wu:2004}. Examples of $\Gamma_{d,d}(1)$-weakly dependent   time series are given in \cite{djellout:guillin:wu:2004}. In particular, ARMA processes with sub-gaussian innovations satisfy the conditions of Theorem \ref{th:wti} for $p=1$, $d=d'=N$. Thus they satisfy Hoeffding's inequality (which is not dimension free).

\subsection{$\tilde \Gamma(p)$-weakly dependent examples}\label{rm:tgp}
We denote $\Gamma_{1,1}(p)=\tilde\Gamma(p)$. For this choice of metrics the optimal coupling scheme is given by the maximal coupling  \cite{goldstein:1979}. We have
$$
\tilde\gamma_{k,i}(p)=\sup_{x^{(i)}, y_i}\|P_{|x^{(i)}}-P_{|y_i,x^{(i-1)}}\|_{TV}^{1/p}
$$
where $\|P-Q\|_{TV}=\sup_A|P(A)-Q(A)|$ for any distributions $P$ et $Q$. The $\tilde\Gamma(p)$ weakly dependent condition coincides with the ones used by Samson \cite{samson:2000} for $p=2$ and by Rio \cite{rio:2000} and Kontorovitch and Ramanan \cite{kontorovich:ramanan:2008}  for $p=1$.\\

For Markov chains, the $\tilde\Gamma(p)$ weakly dependent condition is equivalent to the uniform ergodicity condition. In the stationary case, $\tilde\gamma_{k,i}^p\le 2\phi_{k-i}$ where $\phi$ is the uniform mixing coefficient introduced by Ibragimov \cite{ibragimov:1962}. For $p=2$, we recover the transport inequality obtained by Samson \cite{samson:2000} as any $P_{x_j|x^{(j-1)}}$ satisfies $\tilde T_{2 ,1} (1)$ and by an application of Theorem \ref{th:wti} we obtain
$$
\inf_{\pi\in \tilde M}\Big(\sum_{i=1}^nQ[\pi[X_i\neq Y_i~|~Y_i]^2\Big)^{1/2}\le \|\tilde\Gamma(2)\|_2\sqrt{2\mathcal K(Q|P)}.
$$
Notice that we use here the minimax theorem of Sion and the Proposition 1 of \cite{marton:1996a} to obtain the identity
$$
\inf_{\pi\in \tilde M}\sup_{\alpha_j>0}\frac{\sum_{j=1}^n\pi[\alpha_j(Y)1_{X_j\neq Y_j}]}{(\sum_{j=1}^nQ[\alpha_j(Y)^2])^{1/2}}=\inf_{\pi\in \tilde M}\Big(\sum_{i=1}^nQ[\pi[X_i\neq Y_i~|~Y_i]^2\Big)^{1/2}.
$$
In the stationary case, any $\phi$-mixing processes are $\tilde \Gamma(p)$-weakly dependent with $\|\tilde \Gamma(p)\|_p\le 1+\sum_{i=1}^n(2\phi_i)^{1/p}$ for any $1\le p\le 2$, see \cite{samson:2000}. But the $\tilde \Gamma(p)$-weakly dependence is also satisfied for non stationary sequences, see \cite{kontorovich:ramanan:2008}. However, when $E$ is a real vector space, the choice of the Hamming distance is not natural and the resulting weakly dependent conditions are often too restrictive.

\subsection{$\Gamma_{N,d'}(2)$-weakly dependent exemples}
In what follows, we show that the  choice $d=N$  is natural in many examples in $E=\R^k$. We focus on two generic examples: the Stochastic Recurrent Equations, treated in \cite{djellout:guillin:wu:2004} when $p=1$ only, and the chains with infinite memory  \cite{doukhan:wintenberger:2008}. As an explicit expression of these coefficients is not available, we use the natural coupling provided by the structure of the model to estimate the coefficients $\gamma_{k,i}$.

\begin{exm}[Stochastic Recurrent Equations (SREs)]
Consider the SRE (also called Iterated Random Functions in \cite{duflo:1997} and Random Dynamical Systems in \cite{djellout:guillin:wu:2004}) 
\begin{equation}\label{defsre}
X_0(x) := x \in E,\qquad X_{t+1}(x) = \psi_{t+1}(
X_t(x)),\qquad t\ge  0,
\end{equation}
where $(\psi_t)$ is a sequence of iid random maps. We denote $P$ the probability of the whole process $(\psi_t)_{t\ge 1}$. Assume in the next proposition  that $d$ and $d'$ are any semi-lower continuous metrics satisfying $d\le Md'$ for some $M>0$.
\begin{prop}\label{prop:sre}For any $1\le p\le 2$, if the distribution of $\psi_1(x)$ belongs to $\tilde T_{p,d'}(C)$ or $\tilde T_{p,d'}^{(i)}(C)$ for any $x\in E$ and if there exists some $S>0$ satisfying 
\begin{equation}\label{eq:sre}
\sum_{t=1}^\infty P[d^p(X_t(x),X_t(x'))]^{1/p}\le S d'(x,x') \qquad\forall x,x'\in E.
\end{equation}
then for any $x\in E$ we have that   the law $P^n_x$ of $(X_t(x))_{1\le t\le n}$ on $E^n$  satisfies  $\tilde T_{p,d}(C(M+S)^2n^{2/p-1})$ or $\tilde T_{p,d}^{(i)}(C(M+S)^2n^{2/p-1})$ respectively.
\end{prop}
\begin{proof}
The result is proved by an application of Theorem \ref{th:wti}. The condition of $\Gamma_{d,d'}$-weak dependence is satisfied because the joint law of $(X_t(x),X_t(x))_{t\ge 1}$ is a natural coupling scheme $\pi_{|0}$ of the law of $(X_t)_{t\ge 1}$ given that $(X_0,X_{-1},X_{-2}\ldots)=(x,x_{-1},x_{-2},\ldots)$ and $(X_0,X_{-1},X_{-2}\ldots)=(x',x_{-1},x_{-2},\ldots)$. We obtain similarly natural coupling schemes $\pi_{|i}$ for any $i\ge 0$ and the coefficients $\gamma_{k,i}$ satisfy the relation $\sum_{k>i}\gamma_{k,i}(p)\le S$. The fact that the relation $\sum_{k<n}\gamma_{n,k}(p)\le S$ also holds for any $n\ge 2$ follows from the exchangeability of $(\psi_1,\ldots,\psi_n)$. Using similar arguments than in Remark \ref{rk:st}, we obtain that $\|\Gamma_{d,d'}(p)\|_p\le M+S$. The result is proved as, by the Markov property, $P_{x_j|x^{(j-1)}}$ is the law of $\psi_j(x_{j-1})$ that satisfies $\tilde T_{p,d'}(C)$ or $\tilde T_{p,d'}^{(i)}(C)$ by assumption.
\end{proof}
Let us detail two classical SREs, the ARMA models and the general affine processes when $d=d'=N$. The two first examples cannot be treated optimally by existing results in \cite{djellout:guillin:wu:2004,marton:2004} that use contractive conditions. \end{exm}
\begin{exm}[ARMA  models] Consider the ARMA model
$$
X_0(x) = x,\qquad X_{t+1}(x) = AX_t(x) + \xi_{t+1}
$$
in $E =\R^k$ where $A\in \mathcal M_{k, k}$    (the space of $k \times k$ matrices) and $(\xi_t)$ is a sequence
of  iid random vectors in $\R^k$ called the innovations.  This model is a particular case of the general
model above with $\psi_t(x) = Ax + \xi_t$. The $\Gamma_{N,N}(p)$-weak dependence condition is
equivalent to
$$
\rho_{sp}(A) := \max\{|\lambda|;\mbox{ $\lambda$ is an eigenvalue in $\C$ of $A$} \}< 1,
$$
which is the necessary and sufficient
condition for the ergodicity of this linear ARMA model $(X_t)$. The conditions of Proposition \ref{prop:sre} are satisfied if the law of $\xi_1$ satisfies $\tilde T_{p,N}(C)$ or $\tilde T_{p,N}^{(i)}(C)$. It is in particular the case with $p=2$ for bounded and gaussian innovations $\xi_t$.

The notion of $\Gamma_{N,N}(2)$-weak dependence is more general than the usual mixing ones. For instance, the solution of $X_{t+1}=1/2X_{t}+\xi_{t+1}$ with $\xi_1\sim\mathcal B(1/2)$ is $\Gamma_{N,N}(2)$-weakly dependent but not mixing, see \cite{andrews:1984}.
\end{exm}

\begin{exm}[General affine processes]\label{gap}
Consider now the specific SRE
$$
X_0(x) = x,\quad X_{t+1}(x) = f(X_t(x))+M(X_t(x))\xi_{t+1}\quad \forall t\ge 1,
$$
where $E = \R^k$,
$\xi_t\in \R^{k'}$, $f :$ $\R^k\mapsto \R^k$, $M:$ $\R^k\mapsto \mathcal M_{k,k'}$ (the space of $k\times k'$ matrices) and the
noise $(\xi_t)$ is a sequence of iid random vectors of $\R^{k'}$ such that its distribution $P_\xi$ is centered. Fix $p=2$ and assume that:
\begin{enumerate}
\item[(a)] $P_\xi\in\tilde T_{2,N}(C)$ or $\tilde T_{2,N}^{(i)}(C)$ on $\R^{k'}$;
\item[(b)] $\|M(x) \|  \le K$, $\forall x \in \R^k$, $K>0$, $\|\cdot\|$ denoting also the operator norm on $\mathcal M_{k,k'}$ associated with the euclidian norms on $\R^k$ and $\R^{k'}$;
\item[(c)] the Lyapunov exponent in $L^2$ satisfies
$$
\lambda_{max}(L^2):=\lim_{t\to \infty}\Big(\sup_{x\neq y}\frac{P[\|X_t(x)-X_t(y)\|^2]}{\|x-y\|^2}\Big)^{1/t}<1.
$$
\end{enumerate}
Using a version of Lemma 2.1 in \cite{djellout:guillin:wu:2004} we obtain that conditions (1) and (2) implies that $P_{x_i|x_{i-1}}\in\tilde T_{2,\|\cdot\|}(CK^2)$ or $\tilde T_{2,\|\cdot\|}^{(i)}(CK^2)$. Moreover condition \eqref{eq:sre} is satisfied for some $S>0$ and thus $P_x^n$ satisfies  $\tilde  T_{2,N}(CK^2(1+S )^2)$ or $\tilde T_{2,N}^{(i)}(CK^2(1+S )^2)$ for any $x\in E$.  
\end{exm}

\begin{exm}[Chains with Infinite Memory]
Here assume that $d=d'=d$ is any semi lower-continuous distance. Consider  chains with infinite memory define in \cite{doukhan:wintenberger:2008}  for any function $F:\, E^{ \N}\times \mathcal X\mapsto E$ by the relation:
\begin{equation}\label{cim}
X_t(x)=F(X_{t-1},X_{t-2},\ldots,X_1,x_0,x_{-1},x_{-2},\ldots;\xi_t),\qquad \forall t\ge 1,
\end{equation}
for any  sequence $x=(x_{-t})_{t\ge 0}\in E^{ \N}$ and any iid innovations $\xi_t$ on some measurable space $\mathcal X$.
This model does not satisfy the Markov property. However, it still exists a natural coupling scheme of the law of  $(X_t)_{t\ge 1}$ given that $(X_0,X_{-1},X_{-2}\ldots)=x$ and $(X_0,X_{-1},X_{-2}\ldots)=(y_0,x_{-1},x_{-2},\ldots)$. Indeed, define recursively the trajectory $(Y_t)_{t\ge 1}$ by the relation
$$
Y_t=F(Y_{t-1},Y_{t-2},\ldots,Y_1,y_0,x_{-1},x_{-2},\ldots;\xi_t),\qquad \forall t\ge 1,
$$
where the innovations $(\xi_t)_{t\ge 1}$ are the same than in \eqref{cim}. Then the natural coupling scheme $\pi_{|0}$ is the distribution of $(X_t(x),Y_t)_{t\ge 1}$. Denote $P$ the law of the innovations process $(\xi_t)$ and $P^n_x$ the law of $(X_t(x))_{1\le t\le n}$ on $E^n$, we have
\begin{prop}
Assume there exists a sequence of non negative numbers $(a_i)$ such that $\sum_{i\ge 1}a_i=a<1$, $\sum_{i\ge1}i\log(i)a_i<\infty$ and
\begin{equation}\label{condcim}
P [d(F(x_1,x_2,\ldots;\xi),F(y_1,y_2,\ldots;\xi))^p]^{1/p}\le \sum_{i\ge 1}a_id(x_i,y_i),
\end{equation}
for any $x=(x_1,x_2,\ldots)$ and $y=(y_1,y_2,\ldots)$ in $E^{ \N}$. If the distribution of $F(y;\xi_1)$, $y\in E^\N$, satisfies $\tilde T_{p,d }(C)$ or $\tilde T_{p,d }^{(i)}(C)$ then it exists $C'>0$ such that $P^n_x$ satisfies $\tilde T_{p,d }(C'n^{2/p-1})$ or $\tilde T_{p,d }^{(i)}(C'n^{2/p-1})$ respectively.
\end{prop}
\begin{proof}
Let us compute a bound for the coefficients $P[d^p(X_t(x),Y_t)]^{1/p}$ for all $t\ge 0$ that are estimates of the coefficients $\gamma_{t,0}$. Fix $$\gamma_{1,0}=P [d(F(x_0,x_{-1},\ldots;\xi),F(y_0,x_{-1},\ldots;\xi))^p]^{1/p}.$$ 
Because $a<1$ then $\gamma_{i+t,i}$ is decreasing with $t\ge 1$ for any $i\ge 0$. We have
$$
\gamma_{i+t,i}\le \sum_{j=1}^{t-1}a_j\gamma_{t-j,0},\qquad t,i\ge 1.
$$
Arguments similar than in the proof of Theorem 3.1 in \cite{doukhan:wintenberger:2008} yields
$$
\gamma_{i+t,i}\le \gamma_{1,0}\inf_{1\le p\le t}\Big\{ a^{t/p}+\sum_{j\ge p}a_j\Big\}.
$$
The desired result follows by choosing $p=cr/\log(r)$ such that $\sum_{t\ge1}\gamma_{i+t,i}<\infty$.
\end{proof}
\end{exm}
\begin{exm}[AR($\infty$) models]
As an example of  chains with infinite memory in $E=\R$, consider $(X_t)$ the stationary solution  to the autoregressive equation
$$
X_t=\sum_{i\ge 1}a_iX_{t-i}+\xi_t,\qquad t\in\Z,
$$
where the real numbers $a_i$ are such that $ \sum_{i\ge 1}|a_i|<1$ and $\sum_{i\ge 1}i\log(i)|a_i|<\infty$. Then if $\xi_1$ satisfies $\tilde T_{2,N}(C)$ or $\tilde T_{p,N}^{(i)}(C)$ then the distribution of $(X_1,\ldots,X_n)$ satisfies $\tilde T_{2,N}(C')$ or $ T_{p,N}^{(i)}(C')$, $C'>0$ for any $n\ge 1$. 
\end{exm}
\begin{exm}[General affine processes with infinite memory]
Consider the process on $E=\R^k$ defined as the solution of 
$$
X_t(x)=f(X_{t-1},X_{t-2},X_{t-3}, \ldots)+M(X_{t-1},X_{t-2},X_{t-3}, \ldots)\xi_t,\qquad \forall t\ge 1$$
where $f$ and $M$ are Lipschitz continuous functions with value in $\R^k$ and $\mathcal M(k,k')$ respectively. These general affine models includes classical econometric models and is  estimated in a parametric setting by the quasi maximum likelihood estimator in \cite{bardet:wintenberger:2009}. Denote for $\Psi=f$ and $\Psi=M$ the Lipschitz coefficients
$$
\|\Psi(x)-\Psi(y)\|\le \sum_{i\ge 1}\alpha_i(\Psi)\|x_i-y_i\|, \qquad \forall x,\; y \in E^{  \N}.
$$
If the condition (a) of Example \ref{gap} is satisfied and $\|M(x) \|  \le K$, $\forall x \in E^{ \N}$, $K>0$, $\sum_{i\ge 1}\alpha_i(f)+P_\xi[\xi^2]^{1/2}\alpha_i(M)<1$ and $\sum_{i\ge 1}i\log(i)(\alpha_i(f)+P_\xi[\xi^2]^{1/2}\alpha_i(M))<\infty$ then the distribution of $(X_1,\ldots,X_n)$ satisfies $\tilde T_{2,N}(C')$ or $ T_{2,N}^{(i)}(C')$, $C'>0$, for any $n\ge 1$. 
\end{exm}
\section{Applications to oracle inequalities with fast convergence rates}\label{sec:oi}
In this section, we use the weak transport inequality to obtain new nonexact oracle inequalities in the $\Gamma(2)$-weakly dependent setting and new exact oracle inequalities in the $\tilde\Gamma(2)$-weakly dependent setting. Instead of using the Talagrand concentration inequalities given in Propositions \ref{pr:tal} and \ref{pr:tal2} we prefer to use a more direct approach using conditional weak transport inequalities.

\subsection{The statistical setting}
We focus on oracle inequalities for the the ordinary least square estimator.
Let us consider the case of the linear regression where $X=(Y,Z)=(Y,Z^{(1)},\ldots,Z^{(d)})$ and $E=\R^{ d+1}$. 
The empirical risk is denoted
$$
r(\theta)=\frac1n\sum_{i=1}^n(Y_i-Z_i\theta)^2
$$
where $(X_i)_{1\le i\le n}=(Y_i,Z_i)_{1\le i\le n}$ are the observations and $\theta\in\R^d$ is a parameter that has to be estimated. In our context, these observations are not necessarily independent nor identically distributed and we denote by $P$ their distribution.  The risk of prediction is denoted
$$
R(\theta)=P[r(\theta)]\qquad \forall \theta \in\R^d.
$$
The aim is to estimate the value  $\overline \theta\in\R^d$ such that
$
R(\overline \theta)\le R( \theta)$,  $\forall \theta \in\R^d.
$ 
We consider the Ordinary Least Square (OLS) estimator $\hat\theta$ of $\overline \theta$ such that $r(\hat\theta)\le r(\theta)$ for all $\theta\in\R^d$.  We denote the excess of risk $\overline R(\theta)=R(\theta)-R(\overline \theta)\ge 0$, $\overline r$ its empirical counterpart, $\mathcal Z=(Z_i)_{1\le i\le n}$ the $n\times d$ matrix of the design, $\|\mathcal Z\|^2_n=n^{-1}\sum_{i=1}^n\|Z_i\|^2$ and $G=P[\mathcal Z^T\mathcal Z]$ its corresponding Gram's matrix. Assume that $G$ is a definite positive matrix and denote 
$
\rho=\max(1,\rho_{sp}(G^{-1})).
$
All the results of this sections are given for probability measures $P$ satisfying $T_{2,d}(C)$ and $T_{2,d}^{(i)}(C)$ for some $C>0$ on $E^n$, $n\ge 1$, with $d=N$ or $1_2$. In view of Theorem \ref{th:wti} we focus on the case $p=2$ to get dimension free concentration.   The constant $C$ in the weak transport inequality has to be estimated in each specific statistical case via the $\Gamma_{N,N}(2)$ or $\tilde \Gamma(2)$-weak dependence properties of $(Y_i,Z_i)_{1\le i\le n}$. The case of possibly non linear autoregression is of special interest. There, the vector $Z_i$ is a function of the past values $\varphi(Y_{i-1},\ldots,Y_{1})$ where $\varphi$ can be chosen as the projection on the last coordinates (case of linear autoregression), functions on Fourier basis or wavelets, etc.  If the maximal order of delay $\ell\ge 1$ is fixed, we have $\gamma_{k,0}^{X}(2)\le \gamma^{Y}_{\lceil k/\ell\rceil ,0}(2)$ and in the non-linear autoregressive  case, $\tilde \gamma^{X}_{k,0}(2)\le \|\varphi\|_\infty\tilde \gamma^{Y}_{\lceil k/\ell\rceil ,0}$. Under conditions on the dependence of $Y$, the $\tilde \gamma(2)$ coefficients of $X$ are nicely estimated for any bounded measurable functions $\varphi$ whereas the $\gamma(2)$ coefficients depend also on the regularity of $\varphi$. Thus, a tradeoff is done: $d=N$ is restricted to regular functions $\varphi$ and some sub-gaussian margins but the dependence structure of the observation is vast. Conversely, $d=1$ corresponds to general functions $\varphi$ with no assumptions on the margins but for observations that are nearly independent.

\subsection{Conditional weak transport inequalities}
We recall the classical approach based on the empirical process concentration to motivate the introduction of our new approach. Following \cite{massart:2007}, oracle inequalities will follow from the concentration properties of $\overline r(\hat\theta)$. However, as the distribution of $\hat \theta$ is difficult to deal with, one studies the concentration of the supremum of the empirical process
$$
f(X_1,\ldots,X_n)=\sup_{\theta\in\Theta}\{\overline r(\theta)-\overline R(\theta)\}.
$$
Likely $f$ is a self-bounding function (for $d=1_2$) and one can use the weak transport to extend Bernstein inequalities on $f$ to dependent settings, see \cite{samson:2000}. To obtain oracle inequalities, it remains to study the expected value of the supremum of the empirical process. In an independent context, a classical solution consists in using a chain in argument and an entropy metric approach, see Chapter 13 of \cite{boucheron:lugosi:massart:2013}. In dependent settings, it is not an easy task because the entropy metric depends on the mixing properties, see \cite{rio:1998}. 

Here we take another route that avoid the study of the supremum of the empirical process. Following the PAC-bayesian approach \cite{mcallester:1999,catoni:2004}, the idea is to consider probability measures $\rho_{\theta}$ centered on $\theta$. Then the concentration of properties of $\overline r(\hat\theta)$ will follows from the transport properties of the measure $P\rho_{\hat \theta}$ on $E^n\times \Theta$ equipped with some metric $d$. Notice that $\rho_{\hat \theta}$ is a probability measure defined conditionally to the observations $(X_i)$. Thus, the properties of the measure $P\rho_{\hat\theta}$ are not simple to handle directly. The PAC-bayesian approach consist in introducing artificially the measure $\rho_{\overline \theta}$ called a priori because it does not depend on the observations $(X_i)$. The conditional weak transport approach then extends the transport properties of $P$ for the metric $d_\theta(x,y)=d((x,\theta),(y,\theta))$ to $\rho_{\overline\theta}P$ for $d$. Then we transport $P\rho_{\hat\theta}$ to $Q\rho_{\hat \theta}$, for any $Q$. Finally, considering $f=\overline r-\overline R$ or $f=r-R$, noticing that $f$ has nice "self-bounding" properties for $d_\theta=1_2$ or $d=N$ respectively, we obtain oracle inequalities via $Q\rho_{\hat \theta}[f]$ because $\rho_{\overline\theta}P[f]=0$, see  \eqref{eq:ineq1} for $d_\theta=N$ and \eqref{eq:ineq2} for $d_\theta=1_2$. More formally, for any metric space $\Theta$, we have
\begin{prop}\label{pr:cwt}
If $P$ satisfies $\tilde T_{p,d_\theta}(C_\theta)$ or $\tilde T_{p,d_\theta}^{(i)}(C_\theta)$, $\theta\in\Theta$ then for any $\mu\in M^+(\Theta)$ we have $\mu \otimes P$ that satisfies $\tilde T_{p,d }$ or $\tilde T_{p,d }^{(i)}$ with constant $\mu[C_\theta]$ when $p=1$ and $\sup_\Theta C_\theta$ otherwise.
\end{prop}
\begin{rmk}
The result does not depend on the metric on $\Theta$ nor the transport properties of $\mu$.
\end{rmk}
\begin{proof}
Denote any measure on $E^n\times \Theta$ has $Q\nu$ where $Q\in M^+(E^n)$ and $\nu$ is defined conditionally to $X$. 
From the proof of Theorem \ref{th:E1}, $\tilde T_{p,d }$ is equivalent to the linear inequality
$$
Q\nu [\lambda(f_{\alpha,d}- { \lambda C_\theta}  \alpha^q/q)]-\mu\otimes P[\lambda f]\le (1/p-1/2) C_\theta \lambda^{2}+\mathcal K(Q\nu|\mu\otimes P).
$$
Denote  $Q_\theta$ the conditional probability measure such that $\mu Q_\theta=Q \nu$. From $\tilde T_{p,d_\theta}(C_\theta)$, we infer that
$$
Q_\theta[\lambda(f_{\alpha,d}- { \lambda C_\theta}  \alpha^q/q)]- P[\lambda f]\le (1/p-1/2) C_\theta \lambda^{2}+\mathcal K(Q_\theta |  P).
$$
We obtain the desired result by linearity integrating in $\mu$ and remarking that 
$
\mathcal K(Q\nu|\mu\otimes P)=\mathcal K(\mu Q_\theta |\mu\otimes P)=\mu[\mathcal K( Q_\theta |  P)]
$
\end{proof}

\subsection{Nonexact oracle inequality for $\Gamma_{N,N}(2)$-weakly dependent sequences}
Our first result is a bound on the excess of risk of the OLS estimator for $\Gamma_{N,N}(2)$-weakly dependent observations $X$. Let us first give an oracle inequality that follows from the conditional weak transport described above:
\begin{thm}\label{th:io} Assume that $X$ satisfies $T_{2,N}(C)$ and $T_{2,N}^{(i)}(C)$. For any measure $Q$ and any $\beta>0$ we have
\begin{equation}\label{eq:ci}
Q[\overline R(\hat\theta)]\le Q[\|\mathcal Z\|_n^2]/\beta
+  4 \sqrt{\rho CQ[K]n^{-1}(\mathcal K(Q|P)+ \beta Q[\overline R(\hat\theta)]/2)}
\end{equation}
where
$$
K:=4\frac d\beta+\Big(1+\|\overline \theta\|^2+\frac {d+2}\beta\Big)R(\overline \theta)+\Big(\|\overline \theta\|^2+\frac d\beta\Big)\frac{d-1}{\beta}+(1+\|\overline \theta\|^2)r(\overline \theta).
$$
\end{thm}

\begin{proof}
Considering the change $(\mathcal Z,\theta)\to (\mathcal ZG^{-1/2},G^{1/2}\theta)$, we assume that the Gram matrix $G$ is the identity matrix. This change of variable is $\rho$-Lipschitz function. Thus $\mathcal ZG^{-1/2}$ satisfies $\tilde T_{2,N}(\rho C)$ and $\tilde T_{2,N}^{(i)}(\rho C)$ using similar arguments than in Lemma 2.1 in \cite{djellout:guillin:wu:2004}. In all the sequel, we thus consider $G=I_d$, $\mathcal Z\in \tilde T_{2,N}(\rho C)$ and $\tilde T_{2,N}^{(i)}(\rho C)$. With this notation, $P[\|\mathcal Z\|_n^2]=d$ and $\|\hat\theta-\overline \theta\|^2=R(\hat\theta)-R(\overline \theta)$.\\
We first study the self-bounding properties of  $f=\overline r$: using  the inequality $x^2-y^2\le 2x(x-y)\le 2\|x\|\|x-y\|$ for any $x,y\in \R$ we obtain
\begin{align*}
f(x)-f(x')&\le   \frac1n\sum_{i=1}^n((y_i-z_i \theta)^2-(y_i'-z_i' \theta)^2+(y_i'-z_i' \overline\theta)^2-(y_i-z_i \overline\theta)^2)\\
&\le  \frac2n\sum_{i=1}^n(|y_i-z_i \theta|\|(1,\theta)\|\|x_i-x_i'\|+|y_i'-z_i '\overline\theta|( \|(1,\overline\theta)\|)\|x_i-x_i'\|).
\end{align*}
We apply the conditional weak transport approach.
By definition of $\tilde W_2 $ and using Cauchy-Schwartz inequality we obtain for any $Q_\theta$ conditional on $\theta$ that 
$$
P[f]- Q_\theta[f]\le  2\|(1,\theta)\|\sqrt{n^{-1}R(\theta)   \tilde W_{2}(Q_\theta,P)}
+2\|(1,\overline\theta)\|\sqrt{n^{-1}Q_\theta [r(\overline \theta)]   \tilde W_{2}(P,Q_\theta)}
$$
As $P$ satisfies $\tilde T_{2,N}(\rho C)$ and $\tilde T_{2,N}^{(i)}(\rho C)$ and using the Cauchy-Schwartz inequality  we obtain
$$
Q_\theta[P[f]-f]\le  4\sqrt{\rho Cn^{-1}\mathcal K(Q_\theta|P)((1+\|\theta)\|^2) R(\theta)   
+(1+\|\overline\theta\|^2) Q_\theta[r(\overline \theta)] )}.
$$
Let $\rho_\theta=\mathcal N_d( \theta,\beta^{-1} I_d)$ for any $\beta>0$. Integrating with respect to $\rho_{\overline \theta}$ yields
\begin{align*}
\rho_{\overline \theta}Q_\theta[P[f]-f]
\le&\,  4\rho_{\overline \theta}\Big[\sqrt{\rho Cn^{-1}\mathcal K(Q_\theta|P)( (1+\|\theta\|^2) R(\theta)  
+(1+\|\overline\theta\|^2) Q_\theta[r(\overline \theta)] )}\Big]\\
\le&\,  4 \sqrt{\rho Cn^{-1}\rho_{\overline \theta}[\mathcal K(Q_\theta|P)](\rho_{\overline \theta}[  (1+\|\theta\|^2) R(\theta)  ]  
+(1+\|\overline\theta\|^2 )Q[r(\overline \theta)] )}.
\end{align*}
Choosing $Q_\theta$ such that  $\rho_{\overline \theta} Q_\theta = Q\rho_{\hat\theta}$ we have $\rho_{\overline \theta}[\mathcal K(Q_\theta|P)]=\mathcal K(Q|P)+Q[\mathcal K(\rho_{\hat\theta}|\rho_{\overline\theta})]$ and $\mathcal K(\rho_{\hat\theta}|\rho_{\overline\theta})]\le \beta/2(R(\hat\theta)-R(\overline \theta))$ and
\begin{align*}
&Q\rho_{\hat \theta}[R(\theta)-R(\overline\theta)-r(\theta)+r(\overline\theta)]\le\\
& 4\sqrt{\rho Cn^{-1}(\mathcal K(Q|P)+ \beta/2Q[R(\hat\theta)-R(\overline \theta)]) \rho_{\overline \theta}[  (1+\|\theta)\|^2) R(\theta)  ] 
+(1+\|\overline\theta\|^2) Q[r(\overline \theta)] )}.
\end{align*}
Now, by Jensen's inequality $Q\rho_{\hat \theta}[R(\theta)]\ge Q [R(\hat \theta)]$ and classical computations give that $Q\rho_{\hat \theta}[r(\theta)]\le r(\hat\theta)+Q[\|\mathcal Z\|_n^2]/\beta\le r(\overline \theta)+Q[\|\mathcal Z\|_n^2]/\beta$. Collecting these bounds, we obtain
\begin{align*}
&Q[R(\hat\theta)-R(\overline\theta)-\|\mathcal Z\|_n^2]/\beta]\le \\
&4\sqrt{\rho Cn^{-1}(\mathcal K(Q|P)+ \beta/2Q[R(\hat\theta)-R(\overline \theta)])\rho_{\overline \theta}[  (1+\|\theta)\|^2) R(\theta)  ]  
+(1+\|\overline\theta\|^2 )Q[r(\overline \theta)] )}.
\end{align*}
To end the proof, let us compute $\rho_{\overline \theta}[(1+\| \theta\|^2) R(\theta) ]$ using the following identity
$$
\rho_{\overline \theta}[(1+\| \theta\|^2)  R(\theta) ] =\rho_{\overline \theta}[ R(\theta) ]+\rho_{\overline \theta}[ \| \theta\|^2 ]R(\overline \theta) +\rho_{\overline \theta}[ \| \theta\|^2R(\theta)-R(\overline\theta)].
$$
Let us decompose the last term:
$$
\rho_{\overline \theta}[ \| \theta\|^2R(\theta)-R(\overline\theta)]=\rho_{\overline \theta}[ \| \theta\|^2\|\hat\theta\overline \theta\|]+2n^{-1}P[\mathcal Y\mathcal Z]\rho_{\overline \theta}[ \| \theta\|^2(\theta-\overline \theta)]
$$
where $\mathcal Y=(Y_1,\ldots,Y_n)$. Simple computations on gaussian random variables give 
\begin{align*}
\rho_{\overline \theta}[R(\theta)]& = R(\overline \theta)+d/\beta\\
\rho_{\overline\theta}[\|\theta\|^2]&=\|\overline \theta\|^2+d/\beta\\
\rho_{\overline \theta}[ \| \theta\|^2(\theta-\overline \theta)]&=2\overline \theta/\beta\\
\rho_{\overline \theta}[ \| \theta\|^2\|\theta-\overline\theta\|^2]&=(\|\overline \theta\|^2+d/\beta)(d-1)/\beta+\|\overline \theta\|^2/\beta+3d/\beta.
\end{align*}
The desired result follows collecting all these bounds and noticing that $4P[\mathcal Y\mathcal Z]\overline\theta \le 2n R(\overline\theta)$.
\end{proof}
In the proof above, we obtain the more general result: for any probability measures $\mu$ and $\nu$ such that there exists $Q_\theta$ satisfying $Q\mu=\nu Q_\theta$ we have:
\begin{equation}\label{eq:ineq1}
Q\mu[\overline R]\le Q\mu[\overline r(\theta)]+4\sqrt{\rho C n^{-1}\mathcal K(Q\mu|P\nu)(\nu[(1+\|\theta\|^2)R(\theta)]+(1+\|\overline \theta\|^2)Q[r(\overline \theta)]}.
\end{equation}
Let us discuss the choices $\mu=\rho_{\hat\theta}$ and $\nu=\rho_{\overline\theta}$ made above.  As soon as $\mu$ is centered in $\hat\theta$, Jensen's inequality yields $Q\mu[R(\theta)]\ge Q[R(\hat\theta)]$. Next, if $\mu$ is sufficiently concentrated around $\hat\theta$ then $Q\mu[r(\theta)-r(\overline\theta)]$ is small as $r(\hat\theta)-r(\overline \theta)<0$. Choosing $\mu$ as the Dirac mass in $\hat\theta$ is excluded by the condition of  existence of some measure $Q_\theta$ satisfying $\nu Q_\theta=Q\mu$. The fact that the support of $\mu$ cannot depend on the observations $(X_i)$ constrain us to choose measures supported on the whole space $\R^d$ in absence of a priori information on $\hat\theta$.  The term $Q\mu[r(\theta)-r(\overline\theta)]$ can be seen as an alternative to the classical entropy metric and VC-dimension approach, see Mc Allester \cite{mcallester:1999}. The measure $\mu$ should be chosen in order to bound this term (and the entropy $\mathcal K(\mu|\nu)$). It leads to Gibbs estimators that are nice alternatives to classical estimators, see Chapter 4 of the textbook of Catoni \cite{catoni:2004} in the iid case,  \cite{alquier:wintenberger:2012,alquier:li:wintenberger} in weakly dependent settings. Here we choose the gaussian measures $\mu=\rho_{\hat\theta}$ and $\nu=\rho_{\overline\theta}$ as in Audibert and Catoni \cite{audibert:catoni:2011} for simplicity because  $\mathcal K(\nu|\mu)=\beta/2\|\hat\theta-\overline \theta\|^2$. This choice leads to estimate the term $Q\mu[r(\theta)-r(\overline\theta)]$ by $Q[\|\mathcal Z\|_n^2]/\beta$. This term can easily be estimated by the sum of $d/\beta$ and a concentration term implying the entropy $\mathcal K(Q|P)$. Thus we obtain  nonexact oracle inequalities:
\begin{cor}\label{cor:io1}
For any $0< \varepsilon<1$ and any $(d+2)/n<\eta<1$ we have with probability $1-\varepsilon$:
$$R(\hat\theta)\le (1+B_1 \eta)R(\overline \theta)+\frac{B_2 d +16\rho C\log(\varepsilon^{-1})}{n\eta}+\frac{B_3}{(n\eta)^2}
$$
where 
\begin{align*}
&B_1=2(3+2\|\overline\theta\|^2+\eta/n),\\
&B_2=2(5+\|\overline \theta\|^2),\\
&B_3=2(d(d-1)+d/n) .
\end{align*}
\end{cor}
\begin{rmk}
This result extends nonexact oracle inequalities as developed in \cite{lecue:mendelson:2012} to a dependent context but for the OLS only. Instead of the Bernstein inequality used in \cite{lecue:mendelson:2012}, only Tsirel'son inequality is used through the choice $d=N$ and thus nonexact oracle inequality holds without any constraint on $\theta\in \R^d$ and for  $\Gamma_{N,N}(2)$ dependent sequence with nice margins. 
\end{rmk}
\begin{proof}
As for any $a,b>0$ we have $2\sqrt{ab}\le  a\lambda+b/\lambda$ for any $\lambda>0$ then from \eqref{eq:ci} we obtain
$$
Q[\overline R(\hat\theta)-{\|Z\|^2}/{\beta}-K  \lambda/n- \beta \overline R(\hat\theta)/(2\lambda)]- \frac{4\rho C \mathcal K(Q|P)}{\lambda}\le 0.
$$
Notice that by definition of $K$ we have
$$
Q[K]=4\frac d\beta+\Big(1+\|\overline \theta\|^2+\frac {d+2}\beta\Big)R(\overline \theta)+\Big(\|\overline \theta\|^2+\frac d\beta\Big)\frac{d-1}{\beta}+(1+\|\overline \theta\|^2)Q[r(\overline \theta)]
$$
by similar arguments than in the proof of Theorem \ref{th:io} we have 
$$
Q[r(\overline \theta)]- R(\overline\theta)\le 2\sqrt{2\rho CR(\overline \theta)n^{-1}\mathcal K(Q|P)} .
$$
As $P[\|Z\|^2]=d$ we obtain similarly
$$
Q[\|Z\|^2]-d\le 2\sqrt{2\rho Cdn^{-1}\mathcal K(Q|P)} .
$$
Collecting these bounds and using the Cauchy-Schwartz inequality we obtain
\begin{align*}
Q[\|Z\|^2/\beta+  \lambda/n r(\overline \theta)]\le& \,d/\beta +  \lambda/n R(\overline \theta)\\
&+4\sqrt{\rho Cn^{-1}(d/\beta^2+(  \lambda/n)^2 R(\overline \theta))\mathcal K(Q|P)}.
\end{align*}
Using again that $2\sqrt{ab}\le  a\lambda+b/\lambda$, choosing $\beta=\lambda=n\eta$ and by definition of $B_1$, $B_2$ and $B_3$ we have
$$
Q[\overline R(\hat\theta) -B_1 \eta R(\overline \theta)-B_2d/(n\eta)-B_3/(n\eta)^2]\le   \frac{16 \rho C\mathcal K(Q|P)}{ n\eta}.
$$
Choose $Q$ as the probability $P$ restricted to the complementary of the event corresponding to the desired oracle inequality that we denote $A$. Then
$$
\frac{16\rho C\log(\varepsilon^{-1})}{n\eta}\le Q[\overline R(\hat\theta) -B_1 \eta R(\overline \theta)-B_2d/(n\eta)-B_3/(n\eta)^2]
$$
Combining these two inequality, we assert that for this specific $Q$ we have
$
-\log(\varepsilon)\le \mathcal K(Q|P).
$ The relative entropy is computed explicitly $\mathcal K(Q|P)=-\log (1-P(A))$ and thus the desired result follows.
\end{proof}

\subsection{Exact oracle inequalities for $\tilde\Gamma(2)$-weakly-dependent sequences}

Let us now give an equivalent of \eqref{eq:ineq1} when $d=d'=1$. Then any $f$ has the self-bounding property $
f(x)-f(y)\le |f(x)|1_{x\neq y}+|f(y)|1_{x\neq y}$. Following the lines of the proof of \eqref{eq:ineq1} with $f=\overline r$ we obtain easily
\begin{equation}\label{eq:ineq2}
Q\mu[\overline R]\le Q\mu[\overline r]+2\sqrt{2\rho C\mathcal K(Q\mu|P\nu)(P\nu[ \overline r^2]+Q\mu[ \overline r^2])}.
\end{equation}
For the specific choice $\mu=\rho_{\hat\theta}$ and $\nu=\rho_{\overline\theta}$ we use computations given in Lemma 1.2 in the supplementary material of \cite{audibert:catoni:2011}: for any $\theta \in\R^d$ we have
$$
\rho_{\theta }[ \overline r^2]\le  5\, \overline r(\theta )^2+\frac{4\|\mathcal Z
\|_n^2}{n\beta} r(\theta ) +\frac{4\|\mathcal Z\|_n^4}{n\beta^2}
$$
where $\|\mathcal Z\|_n^4=n^{-1}\sum_{i=1}^n\|Z_i\|^4$.
The quantities $Q[\|\mathcal Z\|_n^2r(\hat\theta)]$ and $Q[\|\mathcal Z\|_n^4]$ can be difficult to estimate for probability measures $Q$. Let us work under a Bernstein assumption that estimates the variance of the excess of risk with its expectation \cite{bartlett:mendelson:2006}. It links the set of parameters $\Theta\subseteq\R^d$ and the support of $P$: there exists some finite constant $B>0$ such that
\begin{equation}\label{eq:span}
B=\sup_{\theta\in\Theta}\frac{\sum_{i=1}^n\|Z_i\theta\|_\infty}{\sum_{i=1}^nP[Z_i\theta]^2}.
\end{equation}
This Bernstein assumption was introduced for the iid   in \cite{audibert:catoni:2011}. Under \eqref{eq:span} and the fact that we assumed $P[\|\mathcal Z\|^2]=d$    we have
$
\|\mathcal Z
\|_n^2\le Bd$ and $\|\mathcal Z
\|_n^2\le (Bd)^2
$. Moreover, using computations given in the supplementary material of  \cite{audibert:catoni:2011} we obtain easily that
$$
\overline r(\theta)^2\le n^{-1}(2B^2+8Br(\overline \theta))\overline R(\theta).
$$
It leads to the following equivalent of Theorem \ref{th:io} 
\begin{thm}\label{th:hd}
If condition \eqref{eq:span} holds, we have
\begin{multline*}
Q[\overline R(\hat\theta)]\le \frac{Bd}{\beta}+2\sqrt{2\rho Cn^{-1}(\mathcal K(Q|P)+\beta Q[\overline R(\hat\theta)]/2)}\times\\
\sqrt{Q[(10B^2+40Br(\overline \theta))\overline R(\hat\theta)]+4Bd(R(\overline \theta)+Q[r(\overline\theta)])/\beta+8(Bd/\beta)^2}.
\end{multline*}
\end{thm}
In the above estimate the terms involving $r(\overline\theta)$ are nuisance terms without additional condition on $\overline \theta$. However, if this term is bounded then the last term of the inequality is proportional to the excess risk $Q[\overline R(\hat\theta)]$. Similarly, in the classical approach  \cite{bartlett:mendelson:2006}, the excess risk also appears in the concentration under the Bernstein condition that controls this variance term by $\overline R(\hat\theta)$. It is the major advantage considering the Hamming distance (Bernstein inequality) compared with the Euclidian distance (Tsirel'son's inequality) where instead of $Q[\overline R(\hat\theta)]$ the term $Q[R(\hat\theta)]$ appeared because $\overline r$ is self-bounding for $1_2$ but only $r$ for $N$. As $Q[\overline R(\hat\theta)]$ is the quantity of interest, we  obtain 
\begin{cor}\label{cor:eoi}
If condition \eqref{eq:span} holds, for any $0< \varepsilon<1$ and any $M>0$ we have with probability $1-\varepsilon$:
\begin{multline*}R(\hat\theta)\le  R(\overline \theta)
+160\frac{B^2+4BM}{n}\times\\\times \Big(Bd + 8\rho C( \log(\varepsilon^{-1}) -\log P(r(\overline \theta)> M))+\frac{d(R(\overline \theta)+M)}{10B+40M}+\frac{8(Bd)^2}{n}\Big)
.\end{multline*}
\end{cor}
\begin{rmk}
Except  \eqref{eq:span}, the exact oracle inequality holds for any $\tilde \Gamma(2)$-weakly dependent sequences without  assumptions on the margins (because any probability measure satisfies $\tilde T_{2,1 }(1)$).  These oracle inequalities are knew, even in the iid case. We refer the reader to  \cite{audibert:catoni:2011} for estimates of the term $\log P(r(\overline \theta)> M)$ in the iid case under finite moments  of order $4$ only.
\end{rmk}
\begin{proof}
Let us denote $A=\{r(\overline \theta)\le M\}$ and $P_A$ the restriction of $P$ on $A$ defined as $P_A(B)=P(B\cap A)$ for any measurable set $B$ on $E^n$. We do not know wether $P_A$ satisfies weak transport inequalities. However, a similar reasoning than to obtain \eqref{eq:ineq2}  yields 
\begin{multline*}
Q[\overline R(\hat\theta)]\le Bd/\beta +\rho_{\hat\theta}\Big[\sqrt{(4Bd\overline R(\overline\theta)/\beta +(4Bd/\beta)^2)n^{-1}  \tilde W_2(Q_\theta,P_A)}\\
+\sqrt{((10B^2+40BM)Q[\overline R(\hat\theta)]+4BdM/\beta +(4Bd/\beta)^2)n^{-1}  \tilde W_2(P_A,Q_\theta)}\Big].
\end{multline*}
We now use the triangular inequality of the weak transport cost \eqref{ti1}:
\begin{align*}
\tilde W_2(P_A,Q_\theta)&\le \tilde W_2(P_A,P)+\tilde W_2(P,Q_\theta),\\
\tilde W_2(Q_\theta,P_A)&\le \tilde W_2(Q_\theta,P)+\tilde W_2(P,P_A).
\end{align*}
Because $P$ satisfies $\tilde T_2(\rho C)$ and $\tilde T_2^{(i)}(\rho C)$, both RHS terms are estimated by 
$$
\sqrt{2\rho C H(P_A|P)}+\sqrt{2\rho C \mathcal K(Q_\theta|P)}\le 4\sqrt{\rho C(\mathcal K(Q_\theta|P)-\log P(A))}
$$
Collecting all these bounds and using the Cauchy-Schwartz inequality we obtain
\begin{multline*}
Q[\overline R(\hat\theta)]\le Bd/\beta + 4\Big[\sqrt{2\rho C n^{-1}( \mathcal K(Q|P)+\beta Q[\overline R(\hat\theta)]/2-\log P(A))}\times\\
\sqrt{((10B^2+40BM)Q[\overline R(\hat\theta)]+4Bd(R(\overline\theta)+M)/\beta +8(Bd/\beta)^2)}.
\end{multline*}
Using several times the inequality $2\sqrt{ab}\le  a\lambda+b/\lambda$ with $\lambda=\beta=n(40B^2+160BM)^{-1}$ yields
$$
Q[\overline R ]/4\le40\frac{B^2+4BM}{n}\Big(Bd + 8\rho C( \mathcal K(Q|P) -\log P(A))+\frac{d(R(\overline \theta)+M)}{10B+40M}+\frac{8(Bd)^2}{n}\Big)
$$
We conclude as in the proof of Corollary \ref{cor:io1} choosing $Q$ as $P$ restricted to the complementary of the event corresponding to the desired oracle inequality.
\end{proof}

\subsubsection*{Acknowledgments}
I thank the GT EVOL's team, especially F. Bolley, I. Gentil, C. Leonard and P.-M. Samson, for presenting and explaining patiently to me the transport inequalities subject. I also thank P. Alquier and G. Lecu\'e for helpful discussions on the PAC-bayesian approach and (non)exact oracle inequalities respectively. I am also grateful to N. Gozlan for indicating me the minimax identity of the weak transport studied here and the one of \cite{marton:1996a}. Finally, I would like to deeply thank  F. Merlev\`ede for pointing out errors in a previous version.

\end{document}